\documentclass{amsart}

\usepackage{graphicx,epsfig}

\usepackage{amsmath,amssymb,latexsym, amsfonts, amscd, amsthm, xy}

\usepackage{url}

\usepackage{hyperref}

\input{xy}
\xyoption{all}

\usepackage[usenames]{color}

\makeindex 

=630pt \headheight = 12pt \headsep = 20pt

\theoremstyle{plain}
\newtheorem{theorem}{Theorem}[section]

\newtheorem{corollary}[theorem]{Corollary}

\newtheorem{lemma}[theorem]{Lemma}

\theoremstyle{definition}

\newtheorem{remark}[theorem]{Remark}
\newtheorem{example}[theorem]{Example}

\def\bA{\mathbb{A}}
\def\bC{\mathbb{C}}
\def\bF{\mathbb{F}}

\def\bP{\mathbb{P}}
\def\bQ{\mathbb{Q}}
\def\bZ{\mathbb{Z}}
\def\bR{\mathbb{R}}

\def\A{\mathbf{A}}

\def\C{\mathbf{C}}
\def\D{\mathbf{D}}

\def\F{\mathbf{F}}
\def\G{\mathbf{G}}
\def\H{\mathbf{H}}
\def\I{\mathbf{I}}
\def\J{\mathbf{J}}

\def\P{\mathbf{P}}

\def\S{\mathbf{S}}

\def\Z{\mathbf{Z}}

\def\cA{\mathcal{A}}
\def\cB{\mathcal{B}}
\def\cC{\mathcal{C}}
\def\cD{\mathcal{D}}

\def\cP{\mathcal{P}}

\def\cQ{\mathcal{Q}}
\def\cS{\mathcal{S}}

\def\cX{\mathcal{X}}
\def\cY{\mathcal{Y}}
\def\cZ{\mathcal{Z}}

\def\cU{\mathcal{U}}
\def\cV{\mathcal{V}}
\def\cT{\mathcal{T}}

\def\Br{\mathrm{Br}}

\begin{document}

\title{Algebraic families of hyperelliptic curves violating the Hasse principle}

\author{Nguyen Ngoc Dong Quan}

\date{July 27, 2014}

\address{Department of Mathematics \\
         University of British Columbia \\
         Vancouver, British Columbia \\
         V6T 1Z2, Canada}

\email{\href{mailto:dongquan.ngoc.nguyen@gmail.com}{\tt dongquan.ngoc.nguyen@gmail.com}}



\maketitle

\tableofcontents

\begin{abstract}

In $2000$, Colliot-Th\'el\`ene and Poonen \cite{colliot-thelene-poonen} showed how to construct algebraic families of genus one curves violating the Hasse principle. Poonen \cite{poonen1} explicitly constructed an algebraic family of genus one cubic curves violating the Hasse principle using the general method developed by Colliot-Th\'el\`ene and himself. The main result in this paper generalizes the result of Colliot-Th\'el\`ene and Poonen to arbitrarily high genus hyperelliptic curves. More precisely, for $n > 5$ and $n \not\equiv 0 \pmod{4}$, we show that there is an algebraic family of hyperelliptic curves of genus $n$ that is counterexamples to the Hasse principle explained by the Brauer-Manin obstruction.

\end{abstract}

\section{Introduction}
\label{Section-Introduction}

The aim of this article is to prove the following result.

\begin{theorem}[Theorem \ref{algebraic-family-hyperelliptic-curve-hp-theorem}]
\label{Theorem-The-main-theorem}

Let $n > 5$ be an integer such that $n \not\equiv 0 \pmod{4}$. Then there is an algebraic family $\cC_t$ of hyperelliptic curves of genus $n$ such that $\cC_t$ is a counterexample to the Hasse principle explained by the Brauer-Manin obstruction for all $t \in \bQ$. Furthermore, $\cC_t$ contains no zero-cycles of odd degree over $\bQ$ for every $t \in \bQ$.

\end{theorem}

Recall from \cite{poonen1} that an \textit{algebraic} family of curves is a family of curves depending on a parameter $T$ such that substituting any rational number for $T$ results in a smooth curve over $\bQ$.

A smooth geometrically irreducible curve $\cC$ over $\bQ$ is said to \textit{satisfy the Hasse principle} if the everywhere local solvability of $\cC$ is equivalent to the global solvability of $\cC$. In more concrete terms, this means that
\begin{align*}
\cC(\bQ) \ne \emptyset \; \; \text{if and only if} \; \; \cC(\bQ_p) \ne \emptyset \; \; \text{for every prime $p$ including $p = \infty$}.
\end{align*}
If $\cC$ has points locally everywhere but has no rational points, we say that $\cC$ is a \textit{counterexample to the Hasse principle}. Furthermore, if we also have $\cC(\bA_{\bQ})^{\Br} = \emptyset$, we say that $\cC$ is a \textit{counterexample to the Hasse principle explained by the Brauer-Manin obstruction}. The Hasse principle fails in general. The first counterexamples of genus one curves to the Hasse principle were discovered by Lind \cite{lind} in $1940$ and independently shortly thereafter by Reichardt \cite{reichardt}. For a basic introduction to the Brauer-Manin obstruction, see \cite{poonen3} and \cite{skorobogatov}.

Let us relate Theorem \ref{Theorem-The-main-theorem} to existing results in literature. For $n = 1$, Colliot-Th\'el\`ene and Poonen \cite{colliot-thelene-poonen} showed how to produce one parameter families of curves of genus one violating the Hasse principle. Poonen \cite{poonen1} explicitly constructed an algebraic family of genus one cubic curves violating the Hasse principle using the general method developed in \cite{colliot-thelene-poonen}. It is not known that there exists an algebraic family of curves of genus $n$ violating the Hasse principle for all $n \ge 2$.

Here, as throughout the article, we say that a smooth geometrically irreducible variety $\cV$ over $\bQ$ is said to \textit{satisfy HP1} if it is a counterexample to the Hasse principle explained by the Brauer-Manin obstruction, and that a smooth geometrically irreducible variety $\cV$ over $\bQ$ is said to \textit{satisfy HP2} if it contains no zero-cycles of odd degree over $\bQ$.

Coray and Manoil \cite{coray-manoil} showed that for each positive integer $n \ge 2$, the smooth projective model of the affine curve defined by
\begin{align}
\label{the-Coray-Manoil-smooth-model-equation}
z^2 = 605\cdot 10^6x^{2n + 2} + (18x^2 - 4400)(45x^2- 8800)
\end{align}
satisfies HP1 and HP2. The Coray-Manoil family of curves is the first family of hyperelliptic curves of varying genus that satisfies HP1 and HP2. Although Coray and Manoil restricted themselves to constructing only one hyperelliptic curve of genus $n$ satisfying HP1 and HP2 for each integer $n \ge 2$, it seems plausible that their approach can be modified to produce algebraic families of hyperelliptic curves of arbitrary genus satisfying HP1 and HP2. Since we will follow the approach of Coray and Manoil with some modifications to prove Theorem \ref{Theorem-The-main-theorem}, we briefly recall their main ideas of constructing the family $(\ref{the-Coray-Manoil-smooth-model-equation})$.

Colliot-Th\'el\`ene, Coray and Sansuc \cite{colliot-thelene-coray-sansuc} proved that the threefold $\cY_{(5, 1, 1)}$ in $\bP^5_{\bQ}$ defined by
\begin{align*}
\cY_{(5, 1, 1)} :
\begin{cases}
u_1^2 - 5v_1^2 &= 2xy \\
u_2^2 - 5v_2^2 &= 2(x + 20y)(x + 25y)
\end{cases}
\end{align*}
satisfies HP1 and HP2. Building on this result, Coray and Manoil \cite{coray-manoil} introduced a geometric construction of hyperelliptic curves that allows to smoothly embed the family of curves defined by $(\ref{the-Coray-Manoil-smooth-model-equation})$ into the threefold $\cY_{(5, 1, 1)}$. It follows immediately from functoriality that the Coray-Manoil family of curves satisfies HP1 and HP2.

In order to generalize the result of Coray and Manoil, we first construct a family of threefolds in $\bP^5_{\bQ}$ that satisfies HP1 and HP2 has the threefold $\cY_{(5, 1, 1)}$ as its member. The construction of such threefolds is achieved building on that of the threefold $\cY_{(5, 1, 1)}$. In order to show that the Brauer-Manin obstruction for these threefolds is non-empty, we also need to show the existence of infinitely many primes $p$ and $q$ satisfying certain quadratic equations. We do this by calling on the result of Iwaniec \cite{iwaniec} that a quadratic polynomial in two variables represents infinitely many primes. Since the existence of certain threefolds in $\bP^5_{\bQ}$ satisfying HP1 and HP2 is of interest in its own right, we state this result here.

\begin{theorem}
\label{Theorem-The-second-main-theorem}

Let $p$ be a prime such that $p \equiv 5 \pmod{8}$ and $3$ is quadratic non-residue in $\bF_p^{\times}$. Then there exist infinitely many couples $(b, d) \in \bZ^2$ such that any smooth and proper $\bQ$-model $\cZ$ of the smooth $\bQ$-variety $\cX$ in $\bA^5_{\bQ}$ defined by
\begin{align*}
\begin{cases}
0 \ne u_1^2 - pv_1^2 &= 2x \\
0 \ne u_2^2 - pv_2^2 &= 2(x + 4pb^2)(x + p^2d^2)
\end{cases}
\end{align*}
satisfies HP1 and HP2.

\end{theorem}

The next step is to choose a family of hyperelliptic curves of arbitrary genus that can be smoothly embed into the family of threefolds in Theorem \ref{Theorem-The-second-main-theorem} using the geometric construction of Coray and Manoil. For each $n \ge 2$, we define a family of hyperelliptic curves of genus $n$ of the shape
\begin{align}
\label{Equation-The-equation-of-hyperelliptic-curves-in-the-introduction}
z^2 = p\alpha^2Q^2 x^{2n + 2} + (2b^2Px^2 + \beta Q)(d^2pPx^2 + 2\beta Q),
\end{align}
where $\alpha, \beta, \gamma$ are certain rational numbers, and $P, Q$ depend on $\alpha, \beta, \gamma, p, b, d$. In order to apply the geometric construction of hyperelliptic curves of Coray and Manoil, the polynomials on the right-hand side of $(\ref{Equation-The-equation-of-hyperelliptic-curves-in-the-introduction})$ are required to be separable.

In order to smoothly embed these hyperelliptic curves into the threefolds in Theorem \ref{Theorem-The-second-main-theorem}, we impose certain conditions on $\alpha, \beta, \gamma$ such that these rational numbers satisfy certain local congruences and certain conics in $\bP^2_{\bQ}$ constructed from sextuples $(p, b, d, \alpha, \beta, \gamma)$ possess at least one non-trivial rational point. Lemma \ref{infinitude-alpha-beta-gamma-lemma} and Lemma \ref{Lemma-the-infinitude-of-p-b-d-alpha-beta-gamma-satisfying-A1-A5-and-B1} show that there are infinitely many sextuples $(p, b, d, \alpha, \beta, \gamma)$ satisfying these conditions. For any such sextuple $(p, b, d, \alpha, \beta, \gamma)$, it follows from functoriality and Theorem \ref{Theorem-The-second-main-theorem} that the family of hyperelliptic curves of genus $n$ defined by $(\ref{Equation-The-equation-of-hyperelliptic-curves-in-the-introduction})$ satisfies HP1 and HP2 for each $n \ge 2$.

In the last step, the main difficulty is to show the existence of rational functions in $\bQ(T)$ that parameterize rational numbers $\alpha, \beta, \gamma$ such that for each integer $n$, substituting any rational number for $T$ in the polynomials on the right-hand side of $(\ref{Equation-The-equation-of-hyperelliptic-curves-in-the-introduction})$ results in a separable polynomial of degree $2n + 2$ over $\bQ$. We do this by calling on a \textit{separability criterion} of the author \cite{dq-separable-polynomials} that will be reviewed in Section \ref{algebraic-families-hyperelliptic-curve-hp-section}.

After this article was finished, the author learned that Bhargava, Gross and Wang \cite{bhargava-gross-wang} showed that for any integer $n \ge 1$, there is a positive proportion of everywhere locally solvable hyperelliptic curves over $\bQ$ of genus $n$ that have no points over any number field of odd degree over $\bQ$. The main theorem of this article describes an explicit algebraic family of such curves of genus $n$ with $\gcd(n, 4) = 1$ and $n > 5$.

\section{The Hasse principle for certain threefolds in $\bP^5_{\bQ}$}
\label{threefold-hp-section}

In this section, we will construct families of threefolds satisfying HP1 and HP2. We begin by stating some lemmas that we will need in the proof of the main results throughout the paper.

\begin{lemma}$(\text{\cite[Lemma 4.8]{coray-manoil}})$
\label{functoriality-azumaya-lemma}

Let $k$ be a number field, and let $\cV_1$ and $\cV_2$ be (proper) $k$-varieties. Assume that there is a $k$-morphism $\alpha : \cV_1 \rightarrow \cV_2$ and $\cV_2(\bA_k)^{\mathrm{Br}} = \emptyset$. Then $\cV_1(\bA_k)^{\mathrm{Br}} = \emptyset$.

\end{lemma}

\begin{lemma}$(\text{\cite[Proposition 6.4]{corn-PLMS}})$
\label{bright-corn-lemma}

Let $\cX$ be a smooth $F$-variety. Let $L/F$ be a cyclic extension, and let $F(\cX)$ be the function field of $\cX$. Let $f$ be an element of $F(\cX)$, and let $\cX_L = \cX \times_F L$. Then the class of the cyclic algebra $\left(L/F, f\right) \in \Br(F(\cX))$ lies in the image of the inclusion $\Br(\cX) \hookrightarrow \Br(F(\cX))$ if and only if $\mathrm{div}(f) = \mathrm{Norm}_{L/F}(D)$ for some $D \in \mathrm{Div}(\cX_L)$.

\end{lemma}

\begin{remark}
\label{Remark-Sketching-a-proof-the-Bright-Corn-lemma-in-the-appendix}

We will sketch a proof of Lemma \ref{bright-corn-lemma} in Subsection \ref{Subsection-Proof-of-the-Bright-Corn-Lemma}.

\end{remark}

\begin{lemma}$(\text{Lang-Nishimura, \cite[Lemme 3.1.1, p. 164]{colliot-thelene-coray-sansuc}})$
\label{lang-nishimura-lemma}

Let $F$ be a field, and let $\cX$ be an integral $F$-variety. Let $\cY$ be a proper $F$-variety, and let $f : \cX \longrightarrow \cY$ be an $F$-rational map. If $\cX(F)$ contains a regular $F$-point, then $\cY(F)$ is non-empty. In particular, the condition $\cX(F) \ne \emptyset$ is an $F$-birational invariant in the category of smooth, proper and integral $F$-varieties $\cX$.

\end{lemma}

The following result describes how to construct certain Azumaya algebras on certain threefolds.

\begin{lemma}
\label{azumaya-threefold-lemma}

Let $p$ be a prime such that $p \equiv 5 \pmod{8}$. Assume that the following are true.
\begin{itemize}

\item [(A1)] $3$ is a quadratic non-residue in $\bF_p^{\times}$.

\item [(B)] there exists a couple $(b, c) \in \bZ^2$ such that $\gcd(b, c) = 1$, $b \not\equiv 0 \pmod{p}$ and $q := |pc - 4b^2|$ is either $1$ or an odd power of an odd prime. Here $|\cdot|$ denotes the absolute value in $\bQ$. Furthermore, if $b \equiv 0 \pmod{3}$, then $c \equiv 2 \pmod{3}$.

\end{itemize}
Let $\cV$ be a smooth, proper $\bQ$-model of the smooth $\bQ$-variety $\cU$ in $\bA^5_{\bQ}$ defined by
\begin{equation}
\label{eqn-U}
\cU:
\begin{cases}
0 \ne u_1^2 - pv_1^2 = 2x \\
0 \ne u_2^2 - pv_2^2 = 2(x + 4pb^2)(x + p^2c).
\end{cases}
\end{equation}
Let $\bQ(\cV)$ be the function field of $\cV$, and let $\cA$ be the class of the quaternion algebra $\left(p, x + 4pb^2\right)$. Then $\cA$ is
an Azumaya algebra of $\cV$, that is, $\cA$ belongs to the subgroup $\Br(\cV)$ of $\Br(\bQ(\cV))$.

\end{lemma}

\begin{proof}

Let $K = \bQ(\sqrt{p})$, and let $\Gamma$ be the divisor defined over $\bQ(\sqrt{p})$ and lying on $\cV$ defined by
\begin{align*}
\Gamma : f := x  + 4pb^2 = 0, \; u_2 - \sqrt{p}v_2 = 0, \; u_1^2 - pv_1^2 = -8pb^2.
\end{align*}
Let $\sigma$ be a generator of $\text{Gal}(K/\bQ)$. We see that
\begin{align*}
\sigma\Gamma : x  + 4pb^2 = 0, \; u_2 + \sqrt{p}v_2 = 0, \; u_1^2 - pv_1^2 = -8pb^2.
\end{align*}
Hence it follows that
\begin{align*}
\mathrm{div}(f) = \Gamma + \sigma \Gamma.
\end{align*}
By Lemma \ref{bright-corn-lemma}, we deduce that $\cA$ is in the image of $\Br(\cV) \hookrightarrow \Br(\bQ(\cV))$.

\end{proof}

\begin{lemma}
\label{existence-of-a}

Let $p$ be a prime such that $p \equiv 5 \pmod{8}$. Assume that conditions $(A1)$ and $(B)$ in Lemma \ref{azumaya-threefold-lemma} are true. Then there exists a non-zero integer $a$ such that
\begin{equation}
\label{gcd-a-b-c}
\gcd\left((a^2 + 2pb^2)(2a^2 + p^2c), 3(2b^2 + pc)\right) = 1.
\end{equation}

\end{lemma}

\begin{proof}

Assume that $H_1 := 2b^2 + pc = \prod_{i = 1}^{m}l_i^{\alpha_i}$, where $l_i$ are distinct primes and $\alpha_i \in \bZ_{>0}$. Note that since $q = |pc - 4b^2|$ is either $1$ or an odd power of an odd prime, $c$ is odd. Thus $H_1$ is odd, and therefore $l_i \ne 2$ for each $1 \le i \le m$. We also have that $l_i \ne p$ for each $1 \le i \le m$; otherwise, $l_i = p$ for some integer $1 \le i \le m$. Since $2b^2 + pc \equiv 0 \pmod{l_i}$ and $l_i = p$, it follows that $b \equiv 0 \pmod{p}$, which is a contradiction. We consider the following cases.

$\star$ \textit{Case 1. $b \equiv 0 \pmod{3}$ and $c \not\equiv 0 \pmod{3}$.}

By assumption $(B)$, one knows that $c \equiv 2 \pmod{3}$. Define $a := \prod_{i = 1}^{m}l_i$. We contend that $a$ satisfies $(\ref{gcd-a-b-c})$. Indeed, we have that $l_i \ne 3$ for each $1 \le i \le m$; otherwise, $l_i = 3$ for some integer $1 \le i \le m$. Since $b \equiv 0 \pmod{3}$ and $p \ne 3$, it follows that $c \equiv 0 \pmod{3}$, which is a contradiction.

Let $H_2 := a^2 + 2pb^2$ and $H_3 := 2a^2 + p^2c$. We see that $a^2 = \prod_{i = 1}^{m}l_i^2 \equiv 1 \pmod{3}$. Since $p \ne 3$, we deduce that
$H_2 \equiv 1 \pmod{3}$ and $H_3 \equiv 2 + c \equiv 1 \pmod{3}$, and thus $H_2H_3 \equiv 1 \pmod{3}$.

Suppose that $l_j$ divides $H_2$ for some integer $1 \le j \le m$. Since $a = \prod_{i = 1}^{m}l_i \equiv 0 \pmod{l_j}$, it follows that
$b \equiv 0 \pmod{l_j}$. Thus $c \equiv 0 \pmod{l_j}$, which is a contradiction to $(B)$.

Suppose that $l_j$ divides $H_3$ for some integer $1 \le j \le m$. Since $a = \prod_{i = 1}^{m}l_i \equiv 0 \pmod{l_j}$ and $l_j \ne p$, it follows that
$c \equiv 0 \pmod{l_j}$. Hence $b \equiv 0 \pmod{l_j}$, which is a contradiction to $(B)$. Therefore, in any event, $(\ref{gcd-a-b-c})$ holds.

$\star$ \textit{Case 2. $b \not\equiv 0 \pmod{3}$ and $c \equiv 0 \pmod{3}$.}

Let $a := \prod_{i = 1}^{m}l_i$. By $(A1)$, we know that $p \equiv 2 \pmod{3}$. Hence repeating in the same manner as in \textit{Case 1}, we deduce that $(\ref{gcd-a-b-c})$ holds.

$\star$ \textit{Case 3. $b \not\equiv 0 \pmod{3}$ and $c \not\equiv 0 \pmod{3}$.}

Let $a := 3\prod_{i = 1}^{m}l_i$. The same arguments as in \textit{Case 1} show that $(\ref{gcd-a-b-c})$ holds.

\end{proof}

Following the techniques in the proof of Proposition 7.1 in \cite{colliot-thelene-coray-sansuc}, we now prove the main theorem in this section.

\begin{theorem}
\label{general-threefold-hp-theorem}

We maintain the same notation as in Lemma \ref{azumaya-threefold-lemma}. Let $p$ be a prime such that $p \equiv 5 \pmod{8}$. Assume further that $(A1)$ and $(B)$ are true. Let $\cU$ and $\cV$ be the $\bQ$-varieties defined by $(\ref{eqn-U})$ in Lemma \ref{azumaya-threefold-lemma}. Let $\cT$ be the singular $\bQ$-variety in $\bP^5_{\bQ}$ defined by
\begin{equation}
\label{eqn-T}
\cT:
\begin{cases}
u_1^2 - pv_1^2 &= 2xy \\
u_2^2 - pv_2^2 &= 2(x + 4pb^2y)(x + p^2cy).
\end{cases}
\end{equation}
Then $\cU$, $\cV$ and $\cT$ satisfy HP1 and HP2.

\end{theorem}

\begin{proof}

$\bullet$ \textit{Step 1. $\cU(\bQ) = \cT(\bQ)$.}

It is clear that $\cU(\bQ) \subseteq \cT(\bQ)$. Hence it suffices to prove that
\begin{align*}
\cT(\bQ) \subseteq \cU(\bQ).
\end{align*}

Assume that there is a point $P := (x : y : u_1 : v_1 : u_2 : v_2) \in \cT(\bQ)$. Suppose first that $y = 0$. It follows from $(\ref{eqn-T})$ that
\begin{equation}
\label{y=0-prove-U-T-equal}
\begin{cases}
u_1^2 - pv_1^2 = 0 \\
u_2^2 - pv_2^2 = 2x^2.
\end{cases}
\end{equation}
We see from the first equation of $(\ref{y=0-prove-U-T-equal})$ that $u_1 = v_1 = 0$. If $x = 0$, then the second equation of $(\ref{y=0-prove-U-T-equal})$ implies that
$u_2 = v_2 = 0$, which is a contradiction. Hence $x \ne 0$, and thus $2 = \left(\dfrac{u_2}{x}\right)^2 - p\left(\dfrac{v_2}{x}\right)^2$. Hence $2$ is the norm of an element in $\bQ(\sqrt{p})^{\times}$, and therefore $2$ is the norm of an element in $\bQ_p(\sqrt{p})^{\times}$. Hence it follows that the local Hilbert symbol $(2, p)_p$ is $1$. On the other hand, using Theorem $5.2.7$ in \cite[page 296]{cohen} and since $p \equiv 5 \pmod{8}$, we deduce that
\begin{align*}
(2, p)_p = \left(\dfrac{2}{p}\right) = -1,
\end{align*}
which is a contradiction.

Now we assume that $y \ne 0$, and with no loss of generality, assume further that $y = 1$. It follows from $(\ref{eqn-T})$ that
\begin{equation}
\label{y=1-prove-U-T-equal}
\begin{cases}
u_1^2 - pv_1^2 = 2x \\
u_2^2 - pv_2^2 = 2(x + 4pb^2)(x + p^2c).
\end{cases}
\end{equation}
We consider the following cases.

$\star$ \textbf{Case 1.} $x = 0$.

The second equation of $(\ref{y=1-prove-U-T-equal})$ implies that $u_2^2 - pv_2^2 = 8p^3b^2c$. Thus $8p^3b^2c$ is the norm of an element in $\bQ_2(\sqrt{p})^{\times}$, and hence the local Hilbert symbol $(8p^3b^2c, p)_2$ is $1$. Since $q = |pc - 4b^2|$ is either $1$ or an odd power of an odd prime, $c$ is odd. Hence $v_2(8p^3b^2c) = 3 + 2v_2(b)$ which is an odd integer. Using Theorem $5.2.7$ in \cite[page 296]{cohen}, we deduce that
\begin{align*}
(8p^3b^2c, p)_2 = \left(\dfrac{p}{2}\right) = -1,
\end{align*}
which is a contradiction.

$\star$ \textbf{Case 2.} $x = -4pb^2$.

It follows from $(\ref{eqn-T})$ that
\begin{equation*}
u_1^2 - pv_1^2 = -8pb^2.
\end{equation*}
Using the same arguments as in \textbf{Case 1}, we deduce that $-8pb^2$ is not the norm of any element in $\bQ_2(\sqrt{p})^{\times}$, which is a contradiction to the last identity.

$\star$ \textbf{Case 3.} $x = -p^2c$.

It follows from $(\ref{eqn-T})$ that
\begin{equation*}
u_1^2 - pv_1^2 = -2p^2c.
\end{equation*}
Using the same arguments as in \textbf{Case 1}, we deduce that $-2p^2c$ is not the norm of any element in $\bQ_2(\sqrt{p})^{\times}$, which is a contradiction to the last identity.

Therefore, in any event, we have shown that if the point $P := (x : y : u_1 : v_1 : u_2 : v_2)$ belongs to $\cT(\bQ)$, then $y = 1$, $x \ne 0$, $x + 4pb^2 \ne 0$ and $x + p^2c \ne 0$. In other words, the point $P$ satisfies
\begin{equation*}
\begin{cases}
0 \ne u_1^2 - pv_1^2 = 2x \\
0 \ne u_2^2 - pv_2^2 = 2(x + 4pb^2)(x + p^2c),
\end{cases}
\end{equation*}
and thus $P \in \cU(\bQ)$. Therefore $\cU(\bQ) = \cT(\bQ)$.

$\bullet$ \textit{Step 2. $\cU, \cV$ and $\cT$ are everywhere locally solvable.}

We now prove that $\cU, \cV$ and $\cT$ are everywhere locally solvable. By Lemma \ref{lang-nishimura-lemma}, it suffices to prove that $\cU$ is everywhere locally solvable. Recall that by  Lemma \ref{existence-of-a}, there is a non-zero integer $a$ such that
\begin{equation*}
\gcd\left((a^2 + 2pb^2)(2a^2 + p^2c), 3(2b^2 + pc)\right) = 1.
\end{equation*}
Hence it suffices to consider the following cases.

$\star$ \textbf{Case I.} $l$ is a prime such that $l \ne p$ and $\gcd(l, (a^2 + 2pb^2)(2a^2 + p^2c)) = 1$.

Let $x = 2a^2$. Since $2x = 4a^2$ is a square in $\bZ$, we see that the local Hilbert symbol $(2x, p)_l$ satisfies
\begin{align*}
(2x, p)_l = (4a^2, p)_l = 1.
\end{align*}
Thus $2x$ is the norm of an element in $\bQ_l(\sqrt{p})^{\times}$. We see that
\begin{equation*}
v_l\left(2(x + 4pb^2)(x + p^2c)\right) = v_l\left(4(a^2 + 2pb^2)(2a^2 + p^2c)\right) = 2v_l(2) + v_l((a^2 + 2pb^2)(2a^2 + p^2c)) = 2v_l(2).
\end{equation*}
Hence using Theorem $5.2.7$ in \cite[page 296]{cohen}, we deduce that the local Hilbert symbol $(2(x + 4pb^2)(x + p^2c), p)_l$ satisfies
\begin{align*}
(2(x + 4pb^2)(x + p^2c), p)_l = 1
\end{align*}
Thus $2(x + 4pb^2)(x + p^2c)$ is the norm of an element in $\bQ_l(\sqrt{p})^{\times}$. Therefore $\cU$ is locally solvable at $l$.

$\star$ \textbf{Case II. } $l$ is a prime such that $\gcd(l,3(2b^2 + pc)) = 1$. Note that $p$ is among these primes.

Assume first that $l = p$, and set $x = 2pb^2$. We see that
\begin{align*}
2x &= p(2b)^2, \\
2(x + 4pb^2)(x + p^2c) &= p^2(12b^2)(2b^2 + pc).
\end{align*}
Note that $(2b)^2 \not\equiv 0 \pmod{p}$ and $(12b^2)(2b^2 + pc) \equiv 6(2b^2)^2 \not\equiv 0 \pmod{p}$. Hence using Theorem $5.2.7$ in \cite[page 296]{cohen}, we deduce that the local Hilbert symbol $(2x, p)_p$ satisfies
\begin{align*}
(2x, p)_p = (-1)^{(p - 1)/2}\left(\dfrac{(2b)^2}{p}\right) = 1.
\end{align*}
Hence $2x$ is the norm of an element in $\bQ_p\left(\sqrt{p}\right)$.

By $(A1)$, we know that $6$ is quadratic residue in $\bF_p^{\times}$. Since $(12b^2)(2b^2 + pc) \equiv 6(2b^2)^2 \pmod{p}$, we see that $(12b^2)(2b^2 + pc)$ is a quadratic residue in $\bF_p^{\times}$. Thus using the same arguments as above, we deduce that
\begin{align*}
\left(2(x + 4pb^2)(x + p^2c), p\right)_p = \left(p^2(12b^2)(2b^2 + pc), p\right)_p = 1.
\end{align*}
Therefore $2(x + 4pb^2)(x + p^2c)$ is the norm of an element in $\bQ_p\left(\sqrt{p}\right)$. Hence $\cU$ is locally solvable at $p$.

Suppose that $l \ne p$, and set $x = 2pb^2$. We see that
\begin{align*}
v_l(2x) &= v_l(4pb^2) = v_l(p) + 2v_l(2b) = 2v_l(2b), \\
v_l(2(x + 4pb^2)(x + p^2c)) &= v_l(p^2(12b^2)(2b^2 + pc)) = 2v_l(2b) + v_l(3(2b^2 + pc)) = 2v_l(2b).
\end{align*}
Hence using the same arguments as in \textbf{Case I}, we deduce that $\cU$ is locally solvable at $l$.

Therefore it follows from \textbf{Case I} and \textbf{Case II} that $\cU$ is everywhere locally solvable, and thus $\cU, \cV$ and $\cT$ are everywhere locally solvable.

$\bullet$ \textit{Step 3. $\cV$ satisfies HP1.}

We will prove that $\cV(\mathbb{A}_{\bQ})^{\Br} = \emptyset$. Let $\bQ(\cV)$ be the function field of $\cV$, and let $\cA$ be the class of quaternion algebra $(p, x + 4pb^2)$ in $\Br(\bQ(\cV))$. It follows from Lemma \ref{azumaya-threefold-lemma} that $\cA$ is an Azumaya algebra of $\cV$. We will prove that for any $P_l \in \cV(\bQ_l)$,
\begin{equation}
\label{inv-value}
 \text{inv}_l(\cA(P_l)) =
\begin{cases}
0 &  \text{if $l \ne 2$,}
\\
1/2 & \text{if $l = 2$.}
\end{cases}
\end{equation}
Since $\cV$ is smooth, we know that $\cU(\bQ_l)$ is $l$-adically dense in $\cV(\bQ_l)$. It is well-known \cite[Lemma 3.2]{viray} that $\text{inv}_l(\cA(P_l))$ is a continuous function on $\cV(\bQ_l)$ with the $l$-adic topology. Hence it suffices to prove $(\ref{inv-value})$ for $P_l \in \cU(\bQ_l)$.

Suppose that $l = \infty$ or $l$ is an odd prime such that $l \ne p$ and $p$ is a square in $\bQ_l^{\times}$. We see that $p \in \bQ_l^{2, \times}$, and hence the local Hilbert symbol $(p, t)_l$ is $1$ for any $t \in \bQ_l^{\times}$. Thus $\text{inv}_l(\cA(P_l))$ is $0$.

Suppose that $l$ is an odd prime such that $l \ne p$ and $p$ is not a square in $\bQ_l^{\times}$. Let $P_l \in \cU(\bQ_l)$, and let $x = x(P_l)$. It follows from equations $(\ref{eqn-U})$ and Theorem $5.2.7$ in \cite[page 296]{cohen} that $v_l(x)$ and $v_l((x + 4pb^2)(x + p^2c))$ are even, and hence the sum $v_l(x + 4pb^2) + v_l(x + p^2c)$ is even. Assume first that $v_l(x) < 0$. We deduce that $v_l(x + 4pb^2) = v_l(x)$, and hence it is even. Suppose now that $v_l(x) \ge 0$. We then see that $v_l(x + 4pb^2) \ge 0$ and $v_l(x + p^2c) \ge 0$. We contend that at least one of the last two numbers is zero. Otherwise, since $x \in \bZ_l$, one sees that $x + 4pb^2 \equiv 0 \pmod{l}$ and $x + p^2c \equiv 0 \pmod{l}$. Hence $l$ divides $p(pc - 4b^2)$, and thus by condition $(B)$, we deduce that $l$ divides $pq$. If $q$ is $1$, then $l = p$, which is a contradiction. Hence with no loss of generality, we might assume that $q$ is an odd power of an odd prime, say $q_1^{2m + 1}$ for some odd prime $q_1$ and $m \in \bZ_{\ge 0}$. It follows that $l = q_1$. By condition $(B)$, we know that $q = q_1^{2m + 1} \equiv \pm 4b^2 \not\equiv 0 \pmod{p}$. Hence
\begin{align*}
l = q_1 \equiv \pm \left(\dfrac{2b}{q_1^m}\right)^2 \pmod{p}.
\end{align*}
Since $-1$ is a square in $\bF_p^{\times}$, it follows from the congruence above that $l$ is a square in $\bF_p^{\times}$. By the quadratic reciprocity law, $p$ is a square in $\bQ_l^{\times}$, which is a contradiction. Since the sum $v_l(x + 4pb^2) + v_l(x + p^2c)$ is even and at least one of the two summands is even, we deduce that each of them is even. Hence using Theorem $5.2.7$ in \cite[page 296]{cohen}, we deduce that the local Hilbert symbol $(p, x + 4pb^2)_l$ is $1$. Therefore $\text{inv}_l(\cA(P_l))$ is $0$.

Suppose that $l = p$. Let $P_p \in \cU(\bQ_p)$ and $x = x(P_p)$. Since the local Hilbert symbol $(p, 2)_p$ is $-1$, we deduce from $(\ref{eqn-U})$ and Theorem $5.2.7$ in \cite[page 296]{cohen} that
\begin{align}
\label{congruence-relation}
\begin{cases}
x = p^n\alpha, \;\text{with} \;n \in \bZ, \alpha \in \bZ_p^{\times} \; \text{and} \; \left(\dfrac{\alpha}{p}\right) = -1, \\
(x + 4pb^2)(x + p^2c) = p^m\beta \; \text{with} \; m \in \bZ, \beta \in \bZ_p^{\times} \; \text{and} \; \left(\dfrac{\beta}{p}\right) = -1.
\end{cases}
\end{align}
Assume that $n \le 0$. We see that $p^{-n}x \equiv \alpha \pmod{p}$. Hence $p^{-n}(x + 4pb^2) \equiv \alpha \pmod{p}$ and
$p^{-n}(x + p^2c) \equiv \alpha \pmod{p}$. Thus the product of the two last congruences contradicts the second equation of $(\ref{congruence-relation})$. Hence with no loss of generality, we might assume that $n \ge 1$. Assume first that $n = 1$. We deduce that $p^{-1}x \equiv \alpha \pmod{p}$, and hence $p^{-1}(x + p^2c) = p^{-1}x + pc \equiv \alpha \pmod{p}$.
Thus by $(\ref{congruence-relation})$, there exists an integer $k \in \bZ$ such that $p^{k}(x + 4pb^2) \equiv \beta \alpha^{-1} \pmod{p}$. We see that $\left(\dfrac{\beta \alpha^{-1}}{p}\right) = 1$. Hence using Theorem $5.2.7$ in \cite[page 296]{cohen}, we deduce that the local Hilbert symbol $\left(p, x + 4pb^2\right)_p$ satisfies
\begin{align*}
\left(p, x + 4pb^2\right)_p = \left(\dfrac{\beta \alpha^{-1}}{p}\right) = 1.
\end{align*}
Therefore $\text{inv}_p(\cA(P_p))$ is $0$.

Suppose now that $n \ge 2$. We see that
\begin{equation*}
p^{-1}(x + 4pb^2) = p^{n - 1}\alpha + 4b^2 \equiv 4b^2 \pmod{p}.
\end{equation*}
Hence using the same arguments as above, we deduce that the local Hilbert symbol $\left(p, x + 4pb^2\right)_p$ is $1$, and thus $\text{inv}_p(\cA(P_p))$ equals $0$.

Therefore, in any event, we see that $\text{inv}_p(\cA(P_p)) = 0$.

Suppose that $l = 2$. Let $P_2 \in \cU(\bQ_2)$, and let $x = x(P_2)$. Since the local Hilbert symbol $(p, 2)_2$ satisfies
\begin{align*}
(p, 2)_2 = \left(\dfrac{p}{2}\right) = -1,
\end{align*}
we deduce from $(\ref{eqn-U})$ and Theorem $5.2.7$ in \cite[page 296]{cohen} that
\begin{align*}
(p, x)_2 = \left(p, (x + 4pb^2)(x + p^2c)\right) = -1.
\end{align*}
Hence $v_2(x)$ and $v_2((x + 4pb^2)(x + p^2c))$ are odd, and thus the sum $v_2(x + 4pb^2) + v_2(x + p^2c)$ is odd. We contend that $v_2(x) \ge 0$. Otherwise, we deduce that
\begin{align*}
v_2(x + 4pb^2) + v_2(x + p^2c) = 2v_2(x),
\end{align*}
which is a contradiction since the left-hand side is odd whereas the right-hand side is even. Since $v_2(x)$ is odd and $v_2(x) \ge 0$, we see that $v_2(x) \ge 1$. Since $c$ is odd, it follows that $v_2(p^2c) = 0$. Hence
$v_2(x + p^2c) = v_2(p^2c) = 0$, and thus $v_2(x + 4pb^2)$ is odd. Since $p \equiv 5 \pmod{8}$, the local Hilbert symbol $(p, x + 4pb^2)_2$ satisfies
\begin{align*}
(p, x + 4pb^2)_2 = \left(\dfrac{p}{2}\right) = -1.
\end{align*}
Therefore $\text{inv}_2(\cA(P_2))$ equals $1/2$.

Therefore, in any event, $\sum_{l}\text{inv}_l\cA(P_l) = 1/2$ for any $(P_l)_{l} \in \cV(\bA_{\bQ})$. Thus $\cV(\bA_{\bQ})^{\Br} = \emptyset$.

$\star$ \textit{Step 4. $\cU$ and $\cT$ satisfy HP1.}

For any point $P_l \in \cU(\bQ_l)$, let $x = x(P_l)$. By the definition of $\cU$, we see that $x + 4pb^2$ is nonzero. By what we have proved in \textit{Step 3}, we know that the local Hilbert symbol $(p, x + 4pb^2)_l$ satisfies
\begin{align*}
(p, x + 4pb^2)_l =
\begin{cases}
1 \; \; &\text{if $l \ne 2$,} \\
- 1\; \; &\text{if $l = 2$.}
\end{cases}
\end{align*}
Hence it follows that $x + 4pb^2$ is the norm of an element of $\bQ_l(\sqrt{p})$ for every $l \ne 2$ including $l = \infty$, and that $x + 4pb^2$ is not a local norm of any element of $\bQ_2(\sqrt{p})$. Thus we deduce that
\begin{align}
\label{Equation-The-product-of-local-Hilbert-symbols-equals-minus-1-in-Theorem-about-no-rational-points-on-the-threefolds}
\prod_{\substack{l}}(p, x + 4pb^2)_l = -1,
\end{align}
where the product is taken over every prime $l$ including $l = \infty$. Therefore it follows from the product formula \cite[Theorem 5.3.1]{cohen} that $\cU(\bQ)$ is empty; otherwise there exists a rational point $P \in \cU(\bQ)$. Thus the element $x + 4pb^2$ is in $\bQ^{\times}$, where $x = x(P)$. Hence by the product formula, we see that
\begin{align*}
\prod_{\substack{l}}(p, x + 4pb^2)_l = 1,
\end{align*}
which is a contradiction to $(\ref{Equation-The-product-of-local-Hilbert-symbols-equals-minus-1-in-Theorem-about-no-rational-points-on-the-threefolds})$. Hence $\cU$ satisfies HP1, and it thus follows from \textit{Step 1} that $\cT$ satisfies HP1.

$\star$ \textit{Step 5. $\cU, \cV$ and $\cT$ satisfy HP2.}

For this contention, note that since $\cT(\bQ) = \emptyset$, it follows from the Amer-Brumer theorem \cite{amer} \cite{brumer} that $\cT$ does not contain any zero-cycle of odd degree over $\bQ$. Thus $\cU, \cV$ and $\cT$ satisfy HP2, and hence our contention follows.

\end{proof}

The following result plays a key role in constructing algebraic families of curves satisfying HP1 and HP2.

\begin{theorem}
\label{threefold-hp-theorem}

Let $p$ be a prime such that $p \equiv 5 \pmod{8}$. Assume $(A1)$, and assume further that the following is true.

\begin{itemize}

\item [(A2)] there exists a couple $(b, d)$ of integers such that $b, d$ are odd, $b \not\equiv 0 \pmod{3}$, $b \not\equiv 0 \pmod{p}$ and $q := |pd^2 - 4b^2|$ is either $1$ or an odd prime.

\end{itemize}
Let $\cZ$ be a smooth and proper $\bQ$-model of the smooth $\bQ$-variety $\cX$ in $\bA^5_{\bQ}$ defined by
\begin{equation}
\label{eqn-X}
\cX :
\begin{cases}
0 \ne u_1^2 - pv_1^2 &= 2x \\
0 \ne u_2^2 - pv_2^2 &= 2(x + 4pb^2)(x + p^2d^2).
\end{cases}
\end{equation}
Let $\cY \subset \bP^5_{\bQ}$ be the singular $\bQ$-variety defined by
\begin{equation}
\label{eqn-Y}
\cY :
\begin{cases}
u_1^2 - pv_1^2 &= 2xy \\
u_2^2 - pv_2^2 &= 2(x + 4pb^2y)(x + p^2d^2y).
\end{cases}
\end{equation}
Then $\cX$, $\cY$ and $\cZ$ satisfy HP1 and HP2.

\end{theorem}

\begin{remark}

In Section \ref{infinitude-threefold-section}, we will prove that there are infinitely many triples $(p, b, d)$ satisfying $(A1)$ and $(A2)$.

\end{remark}

\begin{proof}

Let $c = d^2$. We contend that the couple $(b, c)$ satisfies $(B)$ in Lemma \ref{azumaya-threefold-lemma}. Indeed, we note that $\gcd(b, d) = 1$; otherwise, there exists an odd prime
$l$ such that $b = lb_1$ and $d = ld_1$ for some integers $b_1, d_1 \in \bZ$. Hence $q := l_1^2|pd_1^2 - 4b_1^2|$, which is a contradiction to $(A2)$. Thus $\gcd(b, d) = 1$, and it follows that $\gcd(b, c) = 1$.

We know that $q = |pc - 4b^2|$ is either $1$ or an odd prime, and that $b \not\equiv 0 \pmod{3}$ and $b \not\equiv 0 \pmod{p}$. Hence the couple $(b, c)$ satisfies $(B)$ in Lemma \ref{azumaya-threefold-lemma}. Thus by Theorem \ref{general-threefold-hp-theorem}, we deduce that $\cX$, $\cY$ and $\cZ$ satisfy HP1 and HP2.

\end{proof}

\section{Infinitude of the triples $(p, b, d)$.}
\label{infinitude-threefold-section}

In this section, we will prove that there are infinitely many triples $(p, b, d)$ satisfying $(A1)$ and $(A2)$ in Theorem \ref{threefold-hp-theorem}. We begin by recalling a theorem of Iwaniec's.

Let $P(x,y)$ be a quadratic polynomial in two variables $x$ and $y$. We say that $P$ $\textit{depends essentially on two
variables}$ if $\tfrac{\partial{P}}{\partial{x}}$ and $\tfrac{\partial{P}}{\partial{y}}$ are linearly independent as elements of
the $\bQ$-vector space $\bQ[x, y]$.

\begin{theorem}$(\text{Iwaniec, \cite[p.$443$]{iwaniec}})$
\label{iwaniec-thm}

Let $P(x,y) = ax^2 + bxy + cy^2 + ex + fy + g$ be a quadratic polynomial defined over $\bQ$, and assume that the following are true.

\begin{itemize}

\item[(i)] $a$, $b$, $c$, $e$, $f$, $g$ are in $\bZ$ and $\gcd(a,b,c,e,f,g) = 1$.

\item[(ii)] $P(x,y)$ is irreducible in $\bQ[x,y]$, and represents arbitrarily large odd numbers and depends essentially on two
variables.

\item[(iii)] $D =  af^2 - bef + ce^2 + (b^2 - 4ac)g = 0$ or $\Delta = b^2 - 4ac$ is a perfect square.

\end{itemize}
Then
\begin{equation*}
 Nlog^{-1}N \ll \sum_{\substack{p \le N, \, p = P(x,y)\\ p \; prime}}1.
\end{equation*}

\end{theorem}

We now prove the main lemma in this section.

\begin{lemma}
\label{threefolds-infinitude-triple-lemma}

Let $p$ be a prime such that $p \equiv 5 \pmod{8}$, and assume that $3$ is a quadratic non-residue in $\bF_p^{\times}$. Then there are infinitely many triples $(p, b, d)$ satisfying $(A1)$ and $(A2)$ in Theorem \ref{threefold-hp-theorem}.

\end{lemma}

\begin{proof}

Let $b_0 \in \bZ$ such that $b_0 \not\equiv 0 \pmod{3}$, $b_0 \not\equiv 0 \pmod{p}$ and $b_0$ is odd. Let $P(x, y) \in \bQ[x, y]$ be a polynomial in two variables $x, y$ defined by
\begin{equation*}
P(x, y) := p(2x + 1)^2 - 4(6py + b_0)^2.
\end{equation*}
Expanding $P(x, y)$ in the form of $Ax^2 + Bxy + Cy^2 + Ex + Fy + G$, we know that
\begin{equation*}
A := 4p, \; B := 0, \; C := -144p^2, \; E := 4p, \; F := -48pb_0, \; G := p - 4b_0^2 .
\end{equation*}
Hence we deduce that $AF^2 - BEF + CE^2 + (B^2 - 4AC)G = 0$ and $\gcd(A, B, C, E, F, G) = 1$. Thus $P$ satisfies the conditions in Theorem \ref{iwaniec-thm}. Therefore there are infinitely many primes $q$ such that $q = P(x, y)$. Upon letting $b = 6py + b_0$ and $d = 2x + 1$, we see that the triple $(p, b, d)$ satisfies $(A1)$ and $(A2)$.

\end{proof}

\begin{remark}
\label{Remark-We-only-need-a-special-case-of-Iwaniec-theorem-which-is-due-to-Hecke}

In the proof of Lemma \ref{threefolds-infinitude-triple-lemma}, we only need to use a special case of Iwaniec's theorem that quadratic polynomials in two variables with zero discriminant represent infinitely many primes. As remarked in \cite[page 436]{iwaniec}, the proof of this special case can be obtained by using relatively classical ideas going back to Hecke \cite{hecke-Primes-representations}.

\end{remark}

\begin{example}

Let $(p, b, d) = (5, 1, 1)$. We see that the triple $(p, b, d)$ satisfies $(A1)$ and $(A2)$. Let $\cY_{(5, 1, 1)}$ be the singular $\bQ$-threefold in $\bP^5_{\bQ}$ defined by
\begin{align*}
\cY_{(5, 1, 1)} :
\begin{cases}
u_1^2 - 5v_1^2 = 2xy \\
u_2^2 - 5v_2^2 = 2(x + 20y)(x + 25y).
\end{cases}
\end{align*}
By Theorem \ref{threefold-hp-theorem}, $\cY_{(5, 1, 1)}$ satisfies HP1 and HP2. The threefold $\cY_{(5, 1, 1)}$ is
the well-known Colliot-Th\'el\`ene-Coray-Sansuc threefold \cite[Proposition 7.1, p. 186]{colliot-thelene-coray-sansuc}.

\end{example}

\begin{example}

Let $(p, b, d) = (29, 1, 3)$. We see that
\begin{equation*}
q = |pd^2 - 4b^2| = |29\cdot 3^2 - 4\cdot 1^2| = 257,
\end{equation*}
which is an odd prime. Hence $(29, 1, 3)$ satisfies $(A1)$ and $(A2)$. Let $\cY_{(29, 1, 3)}$ be the singular $\bQ$-threefold in $\bP^5_{\bQ}$ defined by
\begin{align*}
\cY_{(29, 1, 3)} :
\begin{cases}
u_1^2 - 29v_1^2 = 2xy \\
u_2^2 - 29v_2^2 = 2(x + 116y)(x + 7569y).
\end{cases}
\end{align*}
By Theorem \ref{threefold-hp-theorem}, $\cY_{(29, 1, 3)}$ satisfies HP1 and HP2.

\end{example}

\section{Hyperelliptic curves violating the Hasse principle}
\label{hyperelliptic-curve-hp-section}

In this section, we give a sufficient condition under which for each integer $n \ge 2$ and $n \not\equiv 0 \pmod{4}$, there exist hyperelliptic curves of genus $n$ that lie on the threefolds $\cY$ in Theorem \ref{threefold-hp-theorem} and satisfy HP1 and HP2. The sufficient condition is in terms of the existence of certain sextuples $(p, b, d, \alpha, \beta, \gamma)$, and obtained using the geometric construction of hyperelliptic curves due to Coray and Manoil \cite[Proposition 4.2]{coray-manoil}.

\begin{theorem}
\label{hyperelliptic-curves-theorem}

Let $p$ be a prime such that $p \equiv 5 \pmod{8}$, and let $(p, b, d) \in \bZ^3$ be a triple of integers satisfying $(A1)$ and $(A2)$ in Theorem \ref{threefold-hp-theorem}. Let $n$ be an integer such that $n \ge 2$, and let $(\alpha, \beta, \gamma) \in \bQ^3$ be a triple of rational numbers such that $\alpha \beta \gamma \ne 0$. Assume further that the following are true.

\begin{itemize}

\item [(A3)]
\begin{equation}
\label{eqn-P}
P := p\alpha^2 + 2\beta^2 - 2p\gamma^2 \ne 0,
\end{equation}
\begin{equation}
\label{eqn-Q}
Q := 4bdp\gamma - 4b^2\beta - d^2p\beta \ne 0,
\end{equation}
and the conic $\cQ_1 \subset \bP^2_{\bQ}$ defined by
\begin{equation*}
\cQ_1: pU^2 - V^2 - (\beta PQ)T^2 = 0
\end{equation*}
has a point $(u, v, t) \in \bZ^3$ with $uvt \ne 0$ and $\gcd(u, v, t) = 1$.

\item [(S)] the polynomial $P_{p, b, d, \alpha, \beta, \gamma}(x) \in \bQ[x]$ defined by
\begin{align*}
P_{p, b, d, \alpha, \beta, \gamma}(x) := p\alpha^2Q^2 x^{2n + 2} + (2b^2Px^2 + \beta Q)(d^2pPx^2 + 2\beta Q)
\end{align*}
is separable; that is, $P_{p, b, d, \alpha, \beta, \gamma}(x)$ has exactly $2n + 2$ distinct roots in $\bC$.

\end{itemize}
Let $\cC$ be the smooth projective model of the affine curve defined by
\begin{equation}
\label{eqn-hyperelliptic-curve-C}
\cC: z^2 = p\alpha^2Q^2 x^{2n + 2} + (2b^2Px^2 + \beta Q)(d^2pPx^2 + 2\beta Q).
\end{equation}
Then $\cC(\bQ_l) \ne \emptyset$ for every prime $l \ne 2, p$ and $\cC(\bA_{\bQ})^{\Br} = \emptyset$. Furthermore $\cC$ satisfies HP2.

\end{theorem}

\begin{proof}

The proof of Theorem \ref{hyperelliptic-curves-theorem} presented below follows closely from that of Proposition \cite[Proposition 4.2]{coray-manoil}. We begin by recalling the geometric construction of hyperelliptic curves due to Coray and Manoil.

Let $\cC_{a} \subset \bA_K^2$ be the affine curve defined by $z^2 = P(x)$, where $P(x)$ is a separable polynomial of degree $2n + 2$ and $K$ is a number field. Recall from \cite[Chapter II, Exercise 2.14]{silverman-EC} that the smooth projective model of $\cC_{a}$ can be described as the closure of
the image of $\cC_{a}$ under the mapping
\begin{align*}
\cC_{a} &\longrightarrow \bP^{n + 2}_K \\
(x, z) &\mapsto \left(1, x, \ldots, x^{n + 1}, z\right).
\end{align*}
Following \cite[Proposition 4.2]{coray-manoil}, we will index the coordinates of $\bP^{n+2}_K$ in such a way that $z_i$ corresponds to $x^i$ for $0 \le i \le n + 1$ and $z_{n + 2}$ corresponds to $z$.

Using the above arguments, we deduce from $(\ref{eqn-hyperelliptic-curve-C})$ that $\cC$ can be smoothly embedded into the intersection of quadrics defined by
\begin{align}
\label{quadrics-eqns}
\begin{cases}
z_{n + 2}^2 &= p\alpha^2Q^2z_{n+1}^2 + (2b^2Pz_2 + \beta Qz_0)(d^2pPz_2 + 2\beta Qz_0) \\
z_1^2       &= z_2z_0.
\end{cases}
\end{align}

Recall that $(u, v, t) \in \bZ^3$ is the point on the conic $\cQ_1$ defined in $(A3)$ that is assumed to exist. Upon letting
\begin{align*}
\begin{cases}
z_0 &= \dfrac{1}{\beta Q}x \\
z_1 &= \dfrac{t}{u}u_1  \\
z_2 &= \dfrac{2p}{P}y \\
z_{n + 1} &= \dfrac{1}{\alpha Q}v_2 \\
z_{n+2} &= u_2,
\end{cases}
\end{align*}
we deduce from $(\ref{quadrics-eqns})$ that
\begin{align}
\label{dpsurface}
\begin{cases}
\dfrac{\beta PQt^2}{pu^2}u_1^2 &= 2xy \\
u_2^2 - pv_2^2 &= 2(x + 4pb^2y)(x + p^2d^2y).
\end{cases}
\end{align}

We see that $(\ref{dpsurface})$ defines a singular del Pezzo surface $\cD \subseteq \bP^4_{\bQ}$. We contend that $\cD(\bA_{\bQ})^{\text{Br}} = \emptyset$ and $\cD$ does not contain any zero-cycle of odd degree over $\bQ$. Indeed, upon letting
\begin{equation*}
v_1 = \dfrac{v}{pu}u_1,
\end{equation*}
we deduce from the first equation of $(\ref{dpsurface})$ and $(A3)$ that
\begin{equation*}
u_1^2 - pv_1^2 = u_1^2 - p\dfrac{v^2}{p^2u^2}u_1^2 = \dfrac{\beta PQt^2}{pu^2}u_1^2 = 2xy.
\end{equation*}
Therefore $\cD$ is a hyperplane section of the threefold $\cY$ in Theorem \ref{threefold-hp-theorem}. Hence there exists a sequence of $\bQ$-morphisms
\begin{align*}
\cC \longrightarrow \cD \longrightarrow \cY.
\end{align*}
Hence it follows from Lemma \ref{functoriality-azumaya-lemma} and Theorem \ref{threefold-hp-theorem} that $\cD(\bA_{\bQ})^{\text{Br}} = \emptyset$. Thus $\cC(\bA_{\bQ})^{\text{Br}} = \emptyset$. Furthermore since $\cY$ does not contain any zero-cycle of odd degree over $\bQ$, so do $\cC$ and $\cD$.

We now prove that $\cC$ is locally solvable at primes $l$ with $l \ne 2, p$. We consider the following cases.

$\star$ \textbf{Case I.} \textit{$l = \infty$ or $l$ is an odd prime such that $l \ne p$ and $\left(\dfrac{p}{l}\right) = 1$.}

We know that the curve $\cC^{\ast}$ defined by
\begin{equation*}
\cC^{\ast} : z^2 = p\alpha^2Q^2 x^{2n + 2} + y^{2n - 2}(2b^2Px^2 + \beta Qy^2)(d^2pPx^2 + 2\beta Qy^2)
\end{equation*}
is an open subscheme of $\cC$. We see that $P_{\infty} = (x :y : z) = (1 : 0 : \sqrt{p}\alpha Q)$ belongs to $\cC^{\ast}(\bQ_l) \subset \cC(\bQ_l) $, and hence $\cC$ is locally solvable at $l$.

$\star$ \textbf{Case II.} \textit{$l$ is an odd prime such that $\left(\dfrac{2}{l}\right) = 1$.}

It follows from $(\ref{eqn-hyperelliptic-curve-C})$ that the point $P_1 = (x, z) = (0, \sqrt{2}\beta Q)$ belongs to $\cC(\bQ_l)$.

$\star$ \textbf{Case III.} \textit{$l$ is an odd prime such that $l \ne p$ and $\left(\dfrac{2p}{l}\right) = 1$.}

Let $F(x, z)$ be the defining polynomial of $\cC$ defined by
\begin{equation*}
F(x, z) = p\alpha^2Q^2 x^{2n + 2} + (2b^2Px^2 + \beta Q)(d^2pPx^2 + 2\beta Q) - z^2.
\end{equation*}
We see that
\begin{equation*}
F\left(1, \sqrt{2p}(\gamma Q + bdP)\right) = (p\alpha^2Q^2 + 2p(bdP)^2 + 4b^2\beta PQ + \beta pPQd^2 + 2\beta^2 Q^2) - 2p(\gamma Q + bdP)^2
\end{equation*}
Hence it follows from $(\ref{eqn-P})$ and $(\ref{eqn-Q})$ that
\begin{equation*}
p\alpha^2Q^2 + 4b^2\beta PQ + \beta pPQd^2 + 2\beta^2 Q^2 = 2p\gamma^2Q^2 + 4p(\gamma Q)(bdP).
\end{equation*}
Thus
\begin{equation*}
p\alpha^2Q^2 + 2p(bdP)^2 + 4b^2\beta PQ + \beta pPQd^2 + 2\beta^2 Q^2 = 2p(\gamma Q + bdP)^2.
\end{equation*}
Hence we deduce that $F(1,\sqrt{2p}(\gamma Q + bdP)) = 0$, and therefore the point $P_2 = (1,\sqrt{2p}(\gamma Q + bdP))$ belongs to $\cC(\bQ_l)$.

Thus, in any event, $\cC$ is locally solvable at primes $l$ with $l \ne 2, p$, which proves our contention.

\end{proof}

\begin{remark}

Theorem \ref{hyperelliptic-curves-theorem} constructs hyperelliptic curves of genus at least two such that they satisfy HP2 and all conditions in HP1 except local solvability at $2$ and $p$. The rest of this section presents certain sufficient conditions for which those hyperelliptic curves arising from Theorem \ref{hyperelliptic-curves-theorem} are locally solvable at $2$ and $p$, and hence satisfy HP1 and HP2.

\end{remark}

\begin{lemma}
\label{p-divisibility-beta-lemma}

Let $p$ be a prime such that $p \equiv 5 \pmod{8}$, and let $(b, d) \in \bZ^3$ be a couple of integers satisfying $(A1)$ and $(A2)$ in Theorem \ref{threefold-hp-theorem}. Assume that there is a triple $(\alpha, \beta, \gamma) \in \bQ^3$ satisfying $(A3)$ in Theorem \ref{hyperelliptic-curves-theorem}, and assume further that $\alpha, \beta, \gamma \in \bZ_p$. Then there is a rational number $\bar{\beta} \in \bQ$ such that $\beta = p\bar{\beta}$ and $\bar{\beta} \in \bZ_p$.

\end{lemma}

\begin{proof}

Let $\cQ_1$ be the conic defined in $(A3)$. Assume that $(u, v, t) \in \bZ^3$ belongs to $\cQ_1(\bQ)$ such that $uvt \ne 0$ and $\gcd(u, v, t) = 1$. We see that
\begin{equation*}
pu^2 - v^2 - \beta PQt^2 = 0,
\end{equation*}
where $P$ and $Q$ are defined by $(\ref{eqn-P})$ and $(\ref{eqn-Q})$, respectively. Taking the identity above modulo $p$, it follows that
\begin{equation*}
v^2 \equiv 8b^2\beta^4t^2 \pmod{p}.
\end{equation*}
Since $2$ is a quadratic non-residue in $\bF_p^{\times}$ and $b \not\equiv 0 \pmod{p}$, we deduce from the congruence above that
\begin{equation*}
v \equiv \beta t \equiv 0 \pmod{p}.
\end{equation*}
Assume that $\beta \not\equiv 0 \pmod{p}$. Then $v \equiv t \equiv 0 \pmod{p}$, and hence $v = pv_1$ and $t = pt_1$ for some integers $v_1, t_1$. Substituting $v$ and $t$ into the defining equation of the conic $\cQ_1$, we get
\begin{equation*}
u^2 - pv_1^2 - p\beta PQ t_1^2 = 0,
\end{equation*}
and hence it follows that $p$ divides $u$. Thus $p$ divides $\gcd(u, v, t)$, which is a contradiction. Therefore there is a rational number $\bar{\beta} \in \bQ$ such that
$\beta = p\bar{\beta}$ and $\bar{\beta} \in \bZ_p$.

\end{proof}

\begin{remark}
\label{p-divisibility-P-Q-remark}

By Lemma \ref{p-divisibility-beta-lemma}, one knows that if $(\alpha, \beta, \gamma) \in \bQ^3$ satisfies $(A3)$ and $\alpha, \beta, \gamma \in \bZ_p$, then there is a rational number $\bar{\beta}$ such that $\beta = p\bar{\beta}$ and $\bar{\beta} \in \bZ_p$. Hence one sees that $P = pP_1$ and $Q= pQ_1$, where
\begin{align*}
P_1 &:= \alpha^2 + 2p\bar{\beta}^2 - 2\gamma^2, \\
Q_1 &:= 4bd\gamma - 4b^2\bar{\beta} - d^2p\bar{\beta}.
\end{align*}
We also see that $P_1$ and $Q_1$ belong to $\bZ_p$.

\end{remark}

In the proofs of Corollary \ref{hyperelliptic-curves-hp-corollary} and Corollary \ref{n-even-hyperelliptic-curve-hp-corollary} below, we will use Hensel's lemma to deduce local solvability at primes $2$ and $p$. For the sake of self-containedness, we recall the statement of Hensel's lemma.

\begin{theorem}
\label{Theorem-Hensel-lemma}
$(\text{Hensel's lemma, \cite[Theorem 3, Section 5.2]{borevich-shafarevich}})$

Let $p$ be a prime. Let $F(x_1, x_2, \ldots, x_n) \in \bZ_p[x_1, x_2, \ldots, x_n]$ be a polynomial whose coefficients are $p$-adic integers. Let $\delta \ge 0$ be a nonnegative integer. Assume that there are $p$-adic integers $a_1, a_2, \ldots, a_n$ such that for some integer $1 \le k \le n$, we have
\begin{align*}
F(a_1, a_2, \ldots, a_n) &\equiv 0 \pmod{p^{2\delta + 1}}, \\
\dfrac{\partial F}{\partial x_k}(a_1, a_2, \ldots, a_n) &\equiv 0 \pmod{p^{\delta}}, \\
\dfrac{\partial F}{\partial x_k}(a_1, a_2, \ldots, a_n) &\not\equiv 0 \pmod{p^{\delta + 1}}.
\end{align*}
Then there exist $p$-adic integers $\theta_1, \theta_2, \ldots, \theta_n$ such that $F(\theta_1, \theta_2, \ldots, \theta_n) = 0$.

\end{theorem}

The following result provides a sufficient condition under which certain hyperelliptic curves of odd genus satisfy HP1 and HP2.

\begin{corollary}
\label{hyperelliptic-curves-hp-corollary}

We maintain the same notation and assumptions as in Theorem \ref{hyperelliptic-curves-theorem}. Assume $(A1)-(A3)$ and $(S)$ in Theorem \ref{hyperelliptic-curves-theorem}. Assume further that the following are true.

\begin{itemize}

\item [(A4)] $\alpha, \beta, \gamma \in \bZ_2^{\times}$, $\alpha, \gamma, d \in \bZ_p^{\times}$ and $\beta \in \bZ_p$.

\item [(A5)] $\gamma Q_1 + bdP_1 \equiv 0 \pmod{p^2}$, where $\bar{\beta}$, $P_1$ and $Q_1$ are defined as in Remark \ref{p-divisibility-P-Q-remark}.

\item [(A6)] $n \not\equiv -2\left(\dfrac{\gamma}{\alpha}\right)^2 \pmod{p}$, $n \ge 3$ and $n$ is odd.

\end{itemize}
Let $\cC$ be the smooth projective model of the affine curve defined by $(\ref{eqn-hyperelliptic-curve-C})$ in Theorem \ref{hyperelliptic-curves-theorem}. Then $\cC$ satisfies HP1 and HP2.

\end{corollary}

\begin{proof}

By Theorem \ref{hyperelliptic-curves-theorem}, it suffices to prove that $\cC$ is locally solvable at $2$ and $p$.

$\star$ \textit{Step 1. $\cC$ is locally solvable at $p$.}

We will use Theorem \ref{Theorem-Hensel-lemma} with the exponent $\delta = 3$ to prove local solvability of $\cC$ at $p$. We consider the following system of equations
\begin{align}
\label{hensel-p}
\begin{cases}
F(x , z) &= p\alpha^2Q^2 x^{2n + 2} + (2b^2Px^2 + \beta Q)(d^2pPx^2 + 2\beta Q) - z^2 \equiv 0 \pmod{p^7} \\
\dfrac{\partial{F}}{\partial{x}}(x, z) &= (2n + 2)p\alpha^2 Q^2x^{2n + 1} + 4b^2Px(d^2pPx^2 + 2\beta Q) + 2d^2pPx(2b^2Px^2 + \beta Q) \equiv 0 \pmod{p^3} \\
\dfrac{\partial{F}}{\partial{x}}(x, z) &\not\equiv 0 \pmod{p^4}.
\end{cases}
\end{align}
Repeating the same arguments as in \textbf{Case III} of the proof of Theorem \ref{hyperelliptic-curves-theorem}, we deduce that
\begin{equation*}
F(1, 0) = 2p(\gamma Q + bdP)^2.
\end{equation*}
By Remark \ref{p-divisibility-P-Q-remark}, one knows that $P = pP_1$ and $Q = pQ_1$. Hence
\begin{equation}
\label{eqn-F}
F(1, 0) = 2p^3\left(\gamma Q_1 + bdP_1\right)^2.
\end{equation}
Thus it follows from $(A5)$ and $(\ref{eqn-F})$ that $F(1, 0) \equiv 0 \pmod{p^7}$. On the other hand, we see that
\begin{equation}
\label{der-F-eqn}
\dfrac{\partial F}{\partial x}(1, 0) = p^3\left((2n + 2)\alpha^2Q_1^2 + 4b^2P_1(d^2P_1 + 2\bar{\beta}Q_1) + 2d^2P_1(2b^2P_1 + p\bar{\beta}Q_1)\right).
\end{equation}
Since $\alpha, \bar{\beta}, \gamma$ and $P_1, Q_1$ are in $\bZ_p$, one obtains that
\begin{equation*}
\dfrac{\partial F}{\partial x}(1, 0) \equiv 0 \pmod{p^3}.
\end{equation*}

Assume that
\begin{equation}
\label{p^4-der-F-eqn}
\dfrac{1}{p^3}\left(\dfrac{\partial F}{\partial x}(1, 0)\right) \equiv 0 \pmod{p}.
\end{equation}
Since $\gamma \in \bZ_p^{\times}$, it follows from $(A5)$ that
\begin{equation}
\label{relation-PQ}
Q_1 \equiv -\dfrac{bd}{\gamma}P_1 \pmod{p}.
\end{equation}
Upon replacing $Q_1$ by $-\dfrac{bd}{\gamma}P_1$ in $(\ref{p^4-der-F-eqn})$, we deduce that
\begin{equation*}
\dfrac{2P_1^2bd}{\gamma^2}\left((n + 1)\alpha^2bd + \gamma(4bd\gamma - 4b^2\bar{\beta} - d^2p\bar{\beta}) \right) \equiv 0 \pmod{p}.
\end{equation*}
Thus it follows from the equation of $Q_1$ in Remark \ref{p-divisibility-P-Q-remark} that
\begin{equation*}
\dfrac{2P_1^2bd}{\gamma^2}\left((n + 1)\alpha^2bd + \gamma Q_1\right) \equiv 0 \pmod{p}.
\end{equation*}
Note that $P_1 \in \bZ_p^{\times}$; otherwise, we deduce from the equation of $P_1$ in Remark \ref{p-divisibility-P-Q-remark} that
\begin{equation*}
\alpha^2 - 2\gamma^2 \equiv P_1 \equiv 0 \pmod{p}.
\end{equation*}
Since $\alpha, \gamma \in \bZ_p^{\times}$, it follows from the congruence above that $2 \equiv \left(\dfrac{\alpha}{\gamma}\right)^2 \pmod{p}$, which is a contradiction to the fact that
$p \equiv 5 \pmod{8}$. Thus $P_1 \in \bZ_p^{\times}$. Since $2, b, d, \gamma$ and $P_1$ are in $\bZ_p^{\times}$, we obtain that
\begin{equation*}
(n + 1)\alpha^2bd  + \gamma Q_1 \equiv 0 \pmod{p}.
\end{equation*}
Since $\gamma Q_1 \equiv -bdP_1 \pmod{p}$ and $b, d \in \bZ_p^{\times}$, we deduce from the last congruence that
\begin{align*}
(n + 1)\alpha^2 \equiv P_1 \equiv (\alpha^2 - 2\gamma^2) \pmod{p}.
\end{align*}
Since $\alpha, \gamma \in \bZ_p^{\times}$, it follows that $n \equiv -2\left(\dfrac{\gamma}{\alpha}\right)^2 \pmod{p}$, which is a contradiction to $(A6)$. Thus the system $(\ref{hensel-p})$ has a solution $(x, z) = (1, 0)$. By Hensel's lemma, $\cC$ is locally solvable at $p$.

$\star$ \textit{Step 2. $\cC$ is locally solvable at $2$.}

We will use Theorem \ref{Theorem-Hensel-lemma} with the exponent $\delta = 1$ to prove local solvability of $\cC$ at $2$. We consider the following system of equations
\begin{align}
\label{hensel-2}
\begin{cases}
F(x,z) &\equiv 0 \pmod{2^3} \\
\dfrac{\partial{F}}{\partial{x}}(x,z) &\equiv 0 \pmod{2} \\
\dfrac{\partial{F}}{\partial{x}}(x,z) &\not\equiv 0 \pmod{2^2}.
\end{cases}
\end{align}

We see from $(\ref{eqn-F})$ and the equations of $P_1$ and $Q_1$ in Remark \ref{p-divisibility-P-Q-remark} that
\begin{equation*}
F(1, 0) = 2p^3\left(\gamma(4bd\gamma - 4b^2\bar{\beta} - d^2p\bar{\beta}) + bd(\alpha^2 + 2p\bar{\beta}^2 - 2\gamma^2)\right)^2.
\end{equation*}
Since $\beta$ is in $\bZ_2^{\times}$ and $p \ne 2$, we see that $\bar{\beta}$ is also in $\bZ_2^{\times}$. Since $b, d, p, \alpha, \bar{\beta}, \gamma \in \bZ_2^{\times}$, we see that
\begin{equation*}
-d^2p \bar{\beta} \gamma + bd\alpha^2 \equiv 0 \pmod{2}.
\end{equation*}
Let $v_2$ denote the $2$-adic valuation. We see that
\begin{align*}
v_2\left(\gamma(4bd\gamma - 4b^2\bar{\beta} - d^2p\bar{\beta}) + bd(\alpha^2 + 2p\bar{\beta}^2 - 2\gamma^2)\right) &= v_2\left((4\gamma(bd\gamma - b^2\bar{\beta}) + 2bd(p\bar{\beta}^2 - \gamma^2)) + (-d^2p \bar{\beta} \gamma + bd\alpha^2)\right) \\
&\ge \min\left(v_2(4\gamma(bd\gamma - b^2\bar{\beta}) + 2bd(p\bar{\beta}^2 - \gamma^2)), v_2(-d^2p \bar{\beta} \gamma + bd\alpha^2)\right) \\
 &\ge 1.
\end{align*}
Hence $F(1, 0) \equiv 0 \pmod{2^3}$. On the other hand, we know from $(\ref{der-F-eqn})$ that $\dfrac{\partial{F}}{\partial{x}}(1, 0) \equiv 0 \pmod{2}$. Since $n$ is odd, $(2n + 2) \equiv 2(n + 1) \equiv 0 \pmod{2^2}$. Hence it follows from $(\ref{der-F-eqn})$ that
\begin{equation*}
\dfrac{\partial F}{\partial x}(1, 0) \equiv 2d^2p^4\bar{\beta}P_1Q_1 \pmod{2^2}.
\end{equation*}
By $(A4)$ and the equations of $P_1$ and $Q_1$ in Remark \ref{p-divisibility-P-Q-remark}, we know that
\begin{equation*}
d^2p^4\bar{\beta}P_1Q_1 \not\equiv 0 \pmod{2}.
\end{equation*}
Hence we deduce that $\dfrac{\partial F}{\partial x}(1, 0) \not\equiv 0 \pmod{2^2}$. Thus the system $(\ref{hensel-2})$ has a solution $(x, z) = (1, 0)$. By Hensel's lemma, $\cC$ is locally solvable at $2$, and hence our contention follows.

\end{proof}

\begin{remark}
\label{local-solubility-at-2-p-remark}

Assuming $(A1), (A2), (A3), (A5)$ and $(S)$. Following closely the proof of Corollary \ref{hyperelliptic-curves-hp-corollary}, we note that the following are true.

\begin{itemize}

\item [(1)] if $\alpha, \beta, \gamma \in \bZ_2^{\times}$ and $n$ is odd, then $\cC$ is locally solvable at $2$.

\item [(2)] if $\alpha, \gamma, d \in \bZ_p^{\times}$, $\beta \in \bZ_p$, $n \ge 2$ and $n \not\equiv -2\left(\dfrac{\gamma}{\alpha}\right)^2 \pmod{p}$, then $\cC$ is locally solvable at $p$.

\end{itemize}

\end{remark}

We now prove a sufficient condition under which certain hyperelliptic curves of genus $n \equiv 2 \pmod{4}$ satisfy HP1 and HP2.

\begin{corollary}
\label{n-even-hyperelliptic-curve-hp-corollary}

We maintain the same notation as in Theorem \ref{hyperelliptic-curves-theorem} and Corollary \ref{hyperelliptic-curves-hp-corollary}. Assume $(A1)-(A5)$ and $(S)$. Assume further that the following are true.

\begin{itemize}

\item [(B1)] $bd - \bar{\beta}\gamma \equiv 0 \pmod{4}$.

\item [(B2)] $n \not\equiv -2\left(\dfrac{\gamma}{\alpha}\right)^2 \pmod{p}$, $n \ge 2$ and $n \equiv 2 \pmod{4}$.

\end{itemize}
Let $\cC$ be the smooth projective model defined by $(\ref{eqn-hyperelliptic-curve-C})$ in Theorem \ref{hyperelliptic-curves-theorem}. Then $\cC$ satisfies HP1 and HP2.

\end{corollary}

\begin{proof}

By Theorem \ref{hyperelliptic-curves-theorem} and Remark \ref{local-solubility-at-2-p-remark}, it suffices to prove that $\cC$ is locally solvable at $2$. We will use Theorem \ref{Theorem-Hensel-lemma} with the exponent $\delta = 2$ to prove local solvability of $\cC$ at $2$. We consider the following
system of equations
\begin{align}
\label{locally-solvable-2-hensel}
\begin{cases}
F(x , z)  &\equiv 0 \pmod{2^5} \\
\dfrac{\partial{F}}{\partial{x}}(x, z) &\equiv 0 \pmod{2^2} \\
\dfrac{\partial{F}}{\partial{x}}(x, z) &\not\equiv 0 \pmod{2^3},
\end{cases}
\end{align}
where $F(x, z)$ denotes the polynomial in variables $x, z$ defined in $(\ref{hensel-p})$. Since $\alpha \in \bZ_2^{\times}$, we know that $\alpha \equiv 1 \pmod{4}$ or $\alpha \equiv 3 \pmod{4}$. Hence $\alpha^2 \equiv 1 \pmod{4}$. Similarly we know that $\bar{\beta}^2, \gamma^2, b^2, d^2 \equiv 1 \pmod{4}$. Since $p \equiv 5 \pmod{8}$, it follows that
\begin{align*}
P_1 &\equiv 1 \pmod{4}, \\
Q_1 &\equiv -\bar{\beta} \pmod{4}.
\end{align*}
By $(B1)$, we know that
\begin{equation*}
\gamma Q_1 + bd P_1 \equiv bd - \bar{\beta}\gamma \equiv 0 \pmod{4},
\end{equation*}
and hence we deduce from $(\ref{eqn-F})$ that $F(1, 0) \equiv 0 \pmod{2^5}$.

Since $n \equiv 2 \pmod{4}$, there is a non-negative integer $l$ such that $n = 4l + 2$. We know that
\begin{equation*}
4b^2P_1(d^2P_1 + 2\bar{\beta}Q_1) + 2d^2P_1(2b^2P_1 + p\bar{\beta}Q_1) = 8b^2d^2P_1^2 + 8b^2\bar{\beta}P_1Q_1 + 2pd^2\bar{\beta}P_1Q_1.
\end{equation*}
Hence it follows from $(\ref{der-F-eqn})$ that
\begin{equation*}
\dfrac{\partial F}{\partial x}(1, 0) \equiv 2\alpha^2Q_1^2 + 2pd^2\bar{\beta}P_1Q_1 \equiv 2 -2\bar{\beta}^2 \equiv 0 \pmod{2^2}.
\end{equation*}
Similarly one sees that
\begin{equation*}
\dfrac{\partial F}{\partial x}(1, 0) \equiv 5(8l + 6)\alpha^2Q_1^2 + 10pd^2\bar{\beta}P_1Q_1 \pmod{2^3}.
\end{equation*}
Since $\alpha, \bar{\beta}, \gamma, b, d \in \bZ_2^{\times}$, we deduce that $\alpha^2, \bar{\beta}^2, \gamma^2, b^2, d^2 \equiv 1 \pmod{2^3}$. Since $p \equiv 5 \pmod{2^3}$ and $bd\gamma - b^2\bar{\beta} \equiv 0 \pmod{2}$, it follows from the equations of $P_1$ and $Q_1$ in Remark \ref{p-divisibility-P-Q-remark} that
\begin{align*}
P_1 &\equiv 1 \pmod{2^3}, \\
Q_1 &\equiv 4(bd\gamma - b^2\bar{\beta}) - 5\bar{\beta} \equiv -5\bar{\beta} \pmod{2^3}.
\end{align*}
Thus we see that
\begin{equation*}
\dfrac{\partial F}{\partial x}(1, 0) \equiv 30 - 250\bar{\beta}^2 \equiv 4 \not\equiv 0 \pmod{2^3}.
\end{equation*}
Therefore the system $(\ref{locally-solvable-2-hensel})$ has a solution $(x, z) = (1, 0)$. By Hensel's lemma, $\cC$ is locally solvable at $2$, which proves our contention.

\end{proof}

\section{Infinitude of the sextuples $(p, b, d, \alpha, \beta, \gamma)$}
\label{infinitude-hyperelliptic-curve-hp-section}

By Corollary \ref{hyperelliptic-curves-hp-corollary} and Corollary \ref{n-even-hyperelliptic-curve-hp-corollary}, we know that in order to construct algebraic families of hyperelliptic curves satisfying HP1 and HP2, we need to find certain sextuples of rational functions in $\bQ(T)$ that parameterize sextuples $(p, b, d, \alpha, \beta, \gamma)$ satisfying $(A1)-(A5)$, $(S)$ and $(B1)$. In this section, we will show how to produce infinitely many sextuples $(p, b, d, \alpha, \beta, \gamma)$ satisfying $(A1)-(A5)$ and $(B1)$ from the known ones.

\begin{lemma}
\label{infinitude-alpha-beta-gamma-lemma}

Let $(p, b, d)$ be a triple of integers satisfying $(A1)$ and $(A2)$. Assume that there is a triple $(\alpha_0, \beta_0, \gamma_0) \in \bQ^3$ satisfying $(A3)$, $(A4)$, $(A5)$ and $(B1)$.
Let $(u_0, v_0, t_0) \in \bZ^3$ be a point on the conic $\cQ_1^{(\alpha_0, \beta_0, \gamma_0)}$ such that $u_0v_0t_0 \ne 0$ and $\gcd(u_0, v_0, t_0) = 1$, where the conic
$\cQ_1^{(\alpha_0, \beta_0, \gamma_0)}$ is defined by
\begin{equation*}
\cQ_1^{(\alpha_0, \beta_0, \gamma_0)} : pU^2 - V^2 - \beta_0P_0Q_0T^2 = 0
\end{equation*}
with
\begin{align*}
P_0 &= p\alpha_0^2 + 2\beta_0^2 - 2p\gamma_0^2, \\
Q_0 &= 4bdp\gamma_0 - 4b^2\beta_0 - d^2p\beta_0.
\end{align*}
Let $A, B \in \bQ$ be rational numbers, and assume that the following are true.
\begin{itemize}

\item [(C1)] $A, B \in \bZ_2$ and $B^2 - pA^2 \in \bZ_2^{\times}$.

\item [(C2)]  $A \in \bZ_p$ and $B \in \bZ_p^{\times}$.

\item [(C3)] $u := u_0 + AC \ne 0$ and $v := v_0 + BC \ne 0$, where
\begin{equation}
\label{eqn-C}
C := \dfrac{2pu_0A - 2v_0B - 4p^3\alpha_0\beta_0t_0^2Q_0}{B^2 - pA^2 + 4p^5\beta_0t_0^2Q_0}.
\end{equation}

\end{itemize}
Define
\begin{align*}
\begin{cases}
\alpha &:= \alpha_0 + 2p^2C \\
\beta  &:= \beta_0  \\
\gamma &:= \gamma_0.
\end{cases}
\end{align*}
Then the triple $(\alpha, \beta, \gamma) \in \bQ^3$ satisfies $(A3)$, $(A4)$, $(A5)$ and $(B1)$.

\end{lemma}

\begin{remark}
\label{Remark-The-geometric-motivation-for-parametrizing-the-triples-alpha-beta-gamma}

In order to use Theorem \ref{hyperelliptic-curves-theorem} to show the existence of algebraic families of hyperelliptic curves satisfying HP1 and HP2, one of the crucial steps is to describe a parametrization of triples $(\alpha, \beta, \gamma)$ such that the conics associated to these triples in $(A3)$ of Theorem \ref{hyperelliptic-curves-theorem} has a non-trivial rational point. Assuming the existence of one triple $(\alpha_0, \beta_0, \gamma_0)$ satisfying $(A3)-(A5)$ and $(B1)$, Lemma \ref{infinitude-alpha-beta-gamma-lemma} shows how to construct families of triples $(\alpha, \beta, \gamma)$ satisfying the same conditions as the triple $(\alpha_0, \beta_0, \gamma_0)$.

\end{remark}

\begin{proof}

We first prove that $(\alpha, \beta, \gamma)$ satisfies $(A4)$. Since $A \in \bZ_p$, $B \in \bZ_p^{\times}$ and the triple $(\alpha_0, \beta_0, \gamma_0)$ satisfies $(A4)$, it follows that
$B^2 - pA^2 + 4p^5\beta_0t_0^2Q_0 \in \bZ_p^{\times}$. Hence by $(\ref{eqn-C})$ and $(C2)$, we see that $C \in \bZ_p$. Thus $\alpha = \alpha_0 + 2p^2C \in \bZ_p$. Hence it follows that
\begin{equation*}
\alpha \equiv \alpha_0 \not\equiv 0 \pmod{p},
\end{equation*}
which proves that $\alpha \in \bZ_p^{\times}$. By assumption, one knows that the triple $(\alpha_0, \beta_0, \gamma_0)$ satisfies $(A4)$. Since $\beta = \beta_0$ and $\gamma = \gamma_0$, we deduce that $\beta, \gamma \in \bZ_2^{\times}$, $\beta \in \bZ_p$ and $\gamma, d \in \bZ_p^{\times}$. Hence it remains to prove that $\alpha \in \bZ_2^{\times}$. By assumptions and $(C1)$, we know that $Q_0 \in \bZ_2$ and $B^2 - pA^2 \in \bZ_2^{\times}$. Hence it follows that
\begin{align*}
B^2 - pA^2 + 4p^5\beta_0t_0^2Q_0 \equiv B^2 - pA^2 \not\equiv 0 \pmod{2}.
\end{align*}
Thus $B^2 - pA^2 + 4p^5\beta_0t_0^2Q_0 \in \bZ_2^{\times}$, and hence we deduce that $C \in \bZ_2$. Thus we see that
\begin{align*}
\alpha = \alpha_0 + 2p^2C \equiv \alpha_0 \not\equiv 0 \pmod{2}.
\end{align*}
Therefore $\alpha \in \bZ_2^{\times}$, and hence $(\alpha, \beta, \gamma)$ satisfies $(A4)$.

Now we prove that $(\alpha, \beta, \gamma)$ satisfies $(A3)$. By what we have proved above, we know that $\alpha, \beta, \gamma \in \bZ_2^{\times}$. This implies that $\alpha, \beta, \gamma \ne 0$. Let $P$ and $Q$ be the rational numbers defined by $(\ref{eqn-P})$ and $(\ref{eqn-Q})$, respectively. One knows that $Q = Q_0 \ne 0$. Since $\alpha, \beta, \gamma \in \bZ_2^{\times}$, it follows that $P \in \bZ_2$. Hence we deduce that
\begin{equation*}
P \equiv p\alpha^2 \not\equiv 0 \pmod{2},
\end{equation*}
which proves that $P \in \bZ_2^{\times}$. Note that $P \ne 0$ since $P \in \bZ_2^{\times}$.

Let $\cQ_1 \subset \bP^2_{\bQ}$ be the conic defined by
\begin{equation*}
\cQ_1 :pU^2 - V^2 - \beta PQ T^2= 0.
\end{equation*}
We prove that the point $\P := (u, v, t) \in \bQ^3$ belongs to $\cQ_1(\bQ)$, where $u$ and $v$ are defined in $(C3)$ and $t := t_0$. Indeed, since $\beta = \beta_0$, $\gamma = \gamma_0$ and $Q = Q_0$, we deduce from $(\ref{eqn-P})$ that
\begin{align*}
-\beta PQt^2 = -\beta_0t_0^2Q_0(p(\alpha_0 + 2p^2C)^2 + 2\beta_0^2 - 2p\gamma_0^2) = -(4p^5\beta_0t_0^2Q_0)C^2 - (4p^3\alpha_0\beta_0t_0^2Q_0)C - (\beta_0P_0Q_0)t_0^2.
\end{align*}
Hence
\begin{align*}
pu^2 - v^2 -\beta PQt^2 &= p(u_0 + AC)^2 - (v_0 + BC)^2 -(4p^5\beta_0t_0^2Q_0)C^2 - (4p^3\alpha_0\beta_0t_0^2Q_0)C - (\beta_0P_0Q_0)t_0^2 \\
&= \left(pA^2 - B^2 - 4p^5\beta_0t_0^2Q_0 \right)C^2 + \left(2pu_0A - 2v_0B - 4p^3\alpha_0\beta_0t_0^2Q_0\right)C + (pu_0^2 - v_0^2 - \beta_0P_0Q_0t_0^2).
\end{align*}
Since $(u_0, v_0, t_0)$ belongs to $\cQ_1^{(\alpha_0, \beta_0, \gamma_0)}(\bQ)$, we see that
\begin{equation*}
pu_0^2 - v_0^2 - \beta_0P_0Q_0t_0^2 = 0.
\end{equation*}
Hence it follows from from $(\ref{eqn-C})$ that
\begin{align*}
pu^2 - v^2 -\beta PQt^2 = \left(pA^2 - B^2 - 4p^5\beta_0t_0^2Q_0 \right)C^2 + \left(2pu_0A - 2v_0B - 4p^3\alpha_0\beta_0t_0^2Q_0\right)C = 0.
\end{align*}
Thus $\P \in \cQ_1(\bQ)$. Since $\cQ_1$ is a nonsingular conic in $\bP^2_{\bQ}$, $\cQ_1(\bQ) \ne \emptyset$ and $uvt \ne 0$, it follows that $(\alpha, \beta, \gamma)$ satisfies $(A3)$.

We now prove that $(\alpha, \beta, \gamma)$ satisfies $(A5)$. Indeed, we have shown that $(\alpha, \beta, \gamma)$ satisfies $(A3), (A4)$. This implies that $\alpha, \beta, \gamma \in \bZ_p$. By Lemma \ref{p-divisibility-beta-lemma}, we know that there is a rational number $\bar{\beta} \in \bQ$ such that $\beta = p\bar{\beta}$ and $\bar{\beta} \in \bZ_p$. Similarly, since $(\alpha_0, \beta_0, \gamma_0)$ satisfies $(A3)$ and $(A4)$, there is a rational number $\bar{\beta}_0$ such that $\beta_0 = p\bar{\beta}_0$ and $\bar{\beta}_0 \in \bZ_p$. Since $\beta = \beta_0$, we deduce that $\bar{\beta} = \bar{\beta}_0$.

Let $P_1$ and $Q_1$ be the rational numbers defined in Remark \ref{p-divisibility-P-Q-remark} and
let $P_1^{(0)}$ and $Q_1^{(0)}$ be the rational numbers defined by the same equations as $P_1$, $Q_1$, respectively
in Remark \ref{p-divisibility-P-Q-remark} with $(\alpha_0, \bar{\beta_0}, \gamma_0)$ in the role of $(\alpha, \bar{\beta}, \gamma)$. By assumption, one knows that the triple $(\alpha_0, \beta_0, \gamma_0)$ satisfies $(A5)$, that is,
\begin{equation*}
\gamma_0 Q_1^{(0)} + bdP_1^{(0)} \equiv 0 \pmod{p^2}.
\end{equation*}
We will prove that
\begin{equation*}
\gamma Q_1 + bdP_1 \equiv 0 \pmod{p^2}.
\end{equation*}
Indeed, one can check that
\begin{align*}
P_1 = \alpha^2 + 2p\bar{\beta}^2 - 2\gamma^2 = 4p^4C^2 + 4p^2\alpha_0C + P_1^{(0)}
\end{align*}
and $Q_1 = Q_1^{(0)}$. Since $\alpha, \bar{\beta}, \gamma$ are in $\bZ_p$, we deduce that $P_1 \in \bZ_p$. Recall that $C \in \bZ_p$. Hence
\begin{align*}
P_1 = 4p^4C^2 + 4p^2\alpha_0C + P_1^{(0)} \equiv P_1^{(0)} \pmod{p^2},
\end{align*}
and thus we deduce that
\begin{equation*}
\gamma Q_1 + bd P_1 \equiv \gamma_0 Q_1^{(0)} + bdP_1^{(0)} \equiv 0 \pmod{p^2}.
\end{equation*}
Therefore $(\alpha, \beta, \gamma)$ satisfies $(A5)$.

Finally, since $(\alpha_0, \beta_0, \gamma_0)$ satisfies $(B1)$, we see that
\begin{align*}
bd - \bar{\beta}\gamma = bd -\bar{\beta}_0\gamma_0 \equiv 0 \pmod{4}.
\end{align*}
Thus $(\alpha, \beta, \gamma)$ satisfies $(B1)$, which proves our contention.

\end{proof}

\begin{lemma}
\label{infinitude-(A,B)-lemma}

Let $(p, b, d)$ be a triple of integers satisfying $(A1)$ and $(A2)$. Assume that there is a triple $(\alpha_0, \beta_0, \gamma_0) \in \bQ^3$ satisfying $(A3)$, $(A4)$, $(A5)$ and $(B1)$.
Let $(u_0, v_0, t_0) \in \bZ^3$ be a point on the conic $\cQ_1^{(\alpha_0, \beta_0, \gamma_0)}$ such that $u_0v_0t_0 \ne 0$ and $\gcd(u_0, v_0, t_0) = 1$, where $P_0$, $Q_0$ and the conic
$\cQ_1^{(\alpha_0, \beta_0, \gamma_0)}$ are defined as in Lemma \ref{infinitude-alpha-beta-gamma-lemma}. Let $\I$ be the set defined by
\begin{align*}
\I := \left\{(A, B) \in \bQ^2: (A, B) \; \; \text{satisfies} \; \; (C1), (C2), (C3) \; \; \text{in Lemma \ref{infinitude-alpha-beta-gamma-lemma}} \right\}.
\end{align*}
Then $\I$ is of infinite cardinality.

\end{lemma}

\begin{proof}

Let $B_0$ be an integer such that $\gcd(B_0, 2p) = 1$. For each $x \in \bZ$, define $B = 2px + B_0$. We see that $B \in \bZ_2^{\times}$ and $B \in \bZ_p^{\times}$. The latter implies that $B \ne 0$. Let $A = 0$, and let $C$ be the rational number defined by $(\ref{eqn-C})$. Define
\begin{align*}
u &:= u_0 + AC = u_0, \\
v &:= v_0 + BC.
\end{align*}
By assumption, we know that $u = u_0 \ne 0$. Assume that $v = 0$. Since $B \ne 0$, it follows from $(\ref{eqn-C})$ and the definition of $v$ that
\begin{equation*}
C = -\dfrac{v_0}{B} = \dfrac{- 2v_0B - 4p^3\alpha_0\beta_0t_0^2Q_0}{B^2 + 4p^5\beta_0t_0^2Q_0}.
\end{equation*}
Hence we deduce that $B$ is a zero of the quadratic polynomial $\cB(T) \in \bQ[T]$, where $\cB(T)$ is defined by
\begin{equation}
\label{B(T)-polynomial}
\cB(T) := v_0T^2 + (4p^3\alpha_0\beta_0t_0^2Q_0)T - 4p^5\beta_0v_0t_0^2Q_0.
\end{equation}
Hence upon letting $T_1$ and $T_2$ be the zeros of $\cB(T)$, we deduce that $(0, B)$ satisfies $(C3)$ in Lemma \ref{infinitude-alpha-beta-gamma-lemma} if and only if $B \ne T_1$ and $B \ne T_2$. The latter holds if and only if $x \ne \dfrac{T_1  - B_0}{2p}$ and $x \ne \dfrac{T_2  - B_0}{2p}$. This implies that if $T_1, T_2 \not\in \bZ$, then $(0, B)$ automatically satisfies $(C3)$ for any integer $x \in \bZ$. Furthermore we see that $B^2 - pA^2 = B^2 \in \bZ_2^{\times}$. Hence $(0, B)$ satisfies $(C1)$ and $(C2)$. Thus $\J$ is a subset of $\I$, where $\J$ is defined by
\begin{align*}
\J := \left\{(0, B) : x \in \bZ, \;  \; x \ne \dfrac{T_1  - B_0}{2p} \; \; \text{and} \; \; x \ne \dfrac{T_2  - B_0}{2p} \right\}.
\end{align*}
Since $\J$ is of infinite cardinality, so is $\I$. Hence our contention follows.

\end{proof}

Using Lemma \ref{infinitude-alpha-beta-gamma-lemma} and Lemma \ref{infinitude-(A,B)-lemma}, we prove the main result in this section.

\begin{lemma}
\label{Lemma-the-infinitude-of-p-b-d-alpha-beta-gamma-satisfying-A1-A5-and-B1}

There are infinitely many sextuples $(p, b, d, \alpha, \beta, \gamma)$ satisfying $(A1)-(A5)$ and $(B1)$.

\end{lemma}

\begin{proof}

Assume that there is a sextuple $(p, b, d, \alpha_0, \beta_0, \gamma_0)$ satisfying $(A1)-(A5)$ and $(B1)$. Let $(u_0, v_0, t_0) \in \bZ^3$ be a point on the conic $\cQ_1^{(\alpha_0, \beta_0, \gamma_0)}$ such that $u_0v_0t_0 \ne 0$ and $\gcd(u_0, v_0, t_0) = 1$, where $P_0$, $Q_0$ and the conic $\cQ_1^{(\alpha_0, \beta_0, \gamma_0)}$ are defined as in Lemma \ref{infinitude-alpha-beta-gamma-lemma}. Let $\J$ be the set defined in the proof of Lemma \ref{infinitude-(A,B)-lemma}. We construct an infinite sequence $(0, B_n)_{n \in \bZ_{\ge 0}}$ of elements of $\J$ as follows.

Let $(0, B_1)$ be an arbitrary element of $\J$, and assume that the elements $(0, B_i)$ of $\J$ with $1 \le i \le n$ are already constructed.
Since $\J$ is infinite, we can choose an element $(0, B_{n + 1})$ of $\J$ such that $B_{n + 1} \ne B_i$ for $1 \le i \le n$ and  $B_{n + 1}$
is not a zero of any of the polynomials $\H_i(T)$ for $1 \le i \le n$, where for each $1 \le i \le n$, $\H_i(T)$ is defined by
\begin{align}
\label{Hi-polynomial}
\H_i(T) = (v_0B_i + 2p^3\alpha_0\beta_0t_0^2Q_0)T + 2p^3\alpha_0\beta_0t_0^2Q_0B_i - 4p^5\beta_0v_0t_0^2Q_0 \in \bQ[T].
\end{align}
Indeed, we see that $2p^3\alpha_0\beta_0t_0^2Q_0B_i - 4p^5\beta_0v_0t_0^2Q_0 \ne 0$ for every $1 \le i \le n$; otherwise, there is an integer $1 \le i \le n$ such that
\begin{equation*}
\alpha_0B_i = 2p^2v_0.
\end{equation*}
Hence $\alpha_0B_i \not\in \bZ_p^{\times}$, which is a contradiction since $\alpha_0$ and $B_i$ are in $\bZ_p^{\times}$. Hence $\H_i(T)$ is nonzero and of degree at most $1$ for each $1 \le i \le n$. Thus $\H_i(T)$ has at most one zero in $\bZ$ for each $1 \le i \le n$; hence, excluding these $n$ zeros (if existing) and the integers $B_i$ for $1 \le i \le n$ out of the infinite set $\J$, one can choose an element $(0, B_{n+1})$ as desired. Therefore we have inductively constructed an infinite sequence $\left\{(0, B_{n})\right\}_{n \ge 1}$ of elements of $\J$. We contend that for any two distinct members $(0, B_m)$ and $(0, B_n)$ of the sequence with $m < n$, the triples $(\alpha_m, \beta_0, \gamma_0)$ and $(\alpha_n, \beta_0, \gamma_0)$ are distinct, that is, $\alpha_m \ne \alpha_n$, where
\begin{align*}
\alpha_m &:= \alpha_0 + 2p^2C_{(m)}, \\
\alpha_n &:= \alpha_0 + 2p^2C_{(n)},
\end{align*}
and $C_{(m)}$, $C_{(n)}$ are defined as in $(\ref{eqn-C})$ with $(0, B_m)$ and $(0, B_n)$ in the role of $(A, B)$, respectively. Assume the contrary, that is, $\alpha_m = \alpha_n$. It follows that
\begin{equation*}
\dfrac{- 2v_0B_m - 4p^3\alpha_0\beta_0t_0^2Q_0}{B_m^2 + 4p^5\beta_0t_0^2Q_0} = C_{(m)} = C_{(n)} = \dfrac{- 2v_0B_n - 4p^3\alpha_0\beta_0t_0^2Q_0}{B_n^2 + 4p^5\beta_0t_0^2Q_0}.
\end{equation*}
Hence we deduce that
\begin{equation*}
2(B_n - B_m)\left(\left(v_0B_m + 2p^3\alpha_0\beta_0t_0^2Q_0\right)B_n + 2p^3\alpha_0\beta_0t_0^2Q_0B_m - 4p^5\beta_0v_0t_0^2Q_0\right) = 0.
\end{equation*}
Since $B_n \ne B_m$, we deduce that $B_n$ is a zero of $\H_m(T)$, where $\H_m(T)$ is defined by $(\ref{Hi-polynomial})$, which is a contradiction to the choice of $B_n$. Thus we have shown that there are infinitely many sextuples $(p, b, d, \alpha, \beta, \gamma)$ satisfying $(A1)-(A5)$ and $(B1)$ provided that there exists one sextuple $(p, b, d, \alpha_0, \beta_0, \gamma_0)$ satisfying $(A1)-(A5)$ and $(B1)$. On the other hand, in the proof of part $(i)$ of Theorem \ref{algebraic-family-hyperelliptic-curve-hp-theorem} below, we will show that the sextuple
$(p, b, d, \alpha_0, \beta_0, \gamma_0) = (29, 1, 3, 7, 261, 15)$ satisfies $(A1)-(A5)$ and $(B1)$, and hence our contention follows.

\end{proof}

\section{Algebraic families of hyperelliptic curves violating the Hasse principle}
\label{algebraic-families-hyperelliptic-curve-hp-section}

Let $n$ be an integer such that $n > 5$ and $n \not\equiv 0 \pmod{4}$. In this section, using the results in the last section, we will show how to construct algebraic families of hyperelliptic curves of genus $n$ satisfying HP1 and HP2. We begin by proving the following useful lemma.

\begin{lemma}
\label{algebraic-family-hyperelliptic-curve-lemma}

Let $\S$ be a finite set of primes, and let $\G(t) \in \bQ(t)$ be a nonzero rational function. Let $\Z$ be the finite set of rational zeros and poles of $\G(t)$, that is, $\Z$ consists of the rational numbers $z \in \bQ$ for which $\G(z)$ is either zero or infinity. For any $z \in \Z$, let $a_z, b_z$ be integers such that $b_z \ne 0$, $\gcd(a_z, b_z) = 1$ and $z = a_z/b_z$. Assume that the following is true.
\begin{itemize}

\item [(D)] let $z$ be any element in $\Z$ such that $a_z \ne 0$. Then $a_z \not\equiv 0 \pmod{l}$ for each prime $l \in \S$.

\end{itemize}
Then there is a rational function $\F(t) \in \bQ(t)$ such that the following are true.
\begin{itemize}

\item [(1)] $\F(t) \in \bZ_l^{\times}$ for each prime $l \in \S$ and each $t \in \bQ$, and

\item [(2)] $\G(\F(t))$ is defined (that is, not infinity) and nonzero for each $t \in \bQ$.

\end{itemize}

\end{lemma}

\begin{proof}

We consider the following two cases.

$\star$ \textit{Case 1. $\Z$ is non-empty.}

By the Chinese Remainder Theorem, there exists an integer $\epsilon$ such that $\epsilon \equiv 2 \pmod{4}$ and $\epsilon$ is a quadratic non-residue in $\bF_l^{\times}$ for each odd prime $l \in \S$ with $l \ne 2$. Let $p_0$ be an odd prime such that the following are true.

\begin{itemize}

\item [(i)] $p_0 \not\in \S$;

\item [(ii)] $b_z \not\equiv 0 \pmod{p_0}$ for every $z \in \Z$; and

\item [(iii)] let $z$ be any element in $\Z$ such that $a_z \ne 0$. Then $a_z \not\equiv 0 \pmod{p_0}$.

\end{itemize}
For each $z \in \Z$, we define
\begin{equation}
\label{algebraic-families-Di-equation}
D_z := p_0b_z\text{sign}(a_z) \prod_{\substack{w \in \Z\setminus \{z\}}}\max(1, |a_w|) \in \bZ,
\end{equation}
where $\text{sign}(\cdot)$ denotes the usual sign function of $\bR$, that is, $\text{sign}(x) = 1$ if $x \ge 0$, and $\text{sign}(x) = -1$ if $x < 0$. We see that $|D_z| \ge p_0 \ge 3$ for each $z \in \Z$. This implies that $|D_z - 1| \ge 1$ for every $z \in \Z$. We will prove that the rational function $\F(t) \in \bQ(t)$ defined by
\begin{equation}
\label{rational-function-F(t)-nonempty-case}
\F(t) := \left(p_0\prod_{z \in \Z}\max(1, |a_z|)\right)\left(1 + \dfrac{4\left(\prod_{l \in \S, l \ne 2}l \right)\left(\prod_{z \in \Z}(D_z - 1)\right)}{t^2 - p_0^2\epsilon} \right)
\end{equation}
satisfies $(1)$ and $(2)$ in Lemma \ref{algebraic-family-hyperelliptic-curve-lemma}. Indeed, let $t = \dfrac{t_1}{t_2}$, where $t_1, t_2 \in \bZ$, $t_2 \ne 0$ and $\gcd(t_1, t_2) = 1$. For each prime $l$, denote by $v_l$ the $l$-adic valuation of $\bQ_l$. For each prime $l \in \S$ with $l \ne 2$, one knows that
\begin{align*}
v_l\left(\dfrac{1}{t^2 - p_0^2\epsilon}\right) = v_l(t_2^2) - v_l(t_1^2 - t_2^2p_0^2\epsilon).
\end{align*}
Assume that $t_1^2 - t_2^2p_0^2\epsilon \equiv 0 \pmod{l}$. Since $p_0 \ne l$ and $\epsilon$ is a quadratic non-residue in $\bF_l^{\times}$, it follows that $t_1 \equiv t_2 \equiv 0 \pmod{l}$, which is a contradiction. Hence we deduce that $v_l(t_1^2 - t_2^2p_0^2\epsilon) = 0$. Thus we see that
\begin{align*}
v_l\left(\dfrac{1}{t^2 - p_0^2\epsilon}\right) = v_l(t_2^2) \ge 0.
\end{align*}
Therefore $\dfrac{1}{t^2 - p_0^2\epsilon} \in \bZ_l$, and hence we deduce that
\begin{align*}
1 + \dfrac{4\left(\prod_{l \in \S, l \ne 2}l \right)\left(\prod_{z \in \Z}(D_z - 1)\right)}{t^2 - p_0^2\epsilon} \in 1 + l\bZ_l.
\end{align*}
By assumption $(D)$ and the choice of $p_0$, one also knows that $p_0\prod_{z \in \Z}\max(1, |a_z|) \in \bZ_l^{\times}$. Hence it follows that for each prime $l \in \S$ with $l \ne 2$, $\F(t) \in \bZ_l^{\times}$ for every $t \in \bQ$. Thus we have shown that if $2 \not\in \S$, then $\F(t)$ satisfies $(1)$ in Lemma \ref{algebraic-family-hyperelliptic-curve-lemma}. Hence it remains to show that if $2 \in \cS$, then $\F(t) \in \bZ_2^{\times}$ for every $t \in \bQ$.

Let first assume that $t_1$ is even. Hence $t_2$ is odd, and hence one sees that $t_1^2 - t_2^2p_0^2\epsilon \equiv 2 \pmod{4}$. Thus $v_2(t_1^2 - t_2^2p_0^2\epsilon) = 1$. Hence it follows that
\begin{align*}
v_2\left(\dfrac{2}{t^2 - p_0^2\epsilon}\right) = 1 + v_2(t_2^2) - v_2(t_1^2 - t_2^2p_0^2\epsilon) = 0,
\end{align*}
which implies that $\dfrac{2}{t^2 - p_0^2\epsilon} \in \bZ_2$.

Now assume that $t_1$ is odd. Since $\epsilon$ is even, one sees that $t_1^2 - t_2^2p_0^2\epsilon$ is odd. Hence it follows that
\begin{align*}
v_2\left(\dfrac{2}{t^2 - p_0^2\epsilon}\right) = 1 + v_2(t_2^2) - v_2(t_1^2 - t_2^2p_0^2\epsilon) = 1 + v_2(t_2^2) \ge 1.
\end{align*}
Thus we have shown that $\dfrac{2}{t^2 - p_0^2\epsilon} \in \bZ_2$. By the definition of $\F(t)$ and assumption $(D)$, we deduce that $\F(t) \in \bZ_2^{\times}$. Hence $\F(t)$ satisfies $(1)$ in Lemma \ref{algebraic-family-hyperelliptic-curve-lemma}.

Now we prove that $\F(t)$ satisfies $(2)$ in Lemma \ref{algebraic-family-hyperelliptic-curve-lemma}. Since $z$ is a rational zero or pole of $\G(t)$ for each $z \in \Z$, we deduce that $\G(\F(t))$ is defined (that is, not infinity) and nonzero for each $t \in \bQ$ if $\F(t) \ne z$ for every $z \in \Z$ and every $t \in \bQ$.

Assume that there is a rational number $t \in \bQ$ such that $\F(t) = z$ for some $z = a_z/b_z \in \Z$. We consider the following two subcases.

$\bullet$ \textit{Subcase 1. $a_z \ne 0$.}

We see that $\max(1, |a_z|) = |a_z|$. Hence it follows that
\begin{align*}
D_z\left(1 + \dfrac{4\left(\prod_{l \in \S, l \ne 2}l \right)\left(\prod_{w \in \Z}(D_w - 1)\right)}{t^2 - p_0^2\epsilon} \right) = 1.
\end{align*}
Upon multiplying by $t^2 - p_0^2\epsilon$ both sides of the identity above and simplifying, we deduce that
\begin{align*}
t^2 = p_0^2\epsilon - 4D_z\left(\prod_{\substack{l \in \S, l \ne 2}}l \right)\left(\prod_{\substack{w \in \Z\setminus \{z\}}}(D_w - 1)\right).
\end{align*}
Hence it follows from $(\ref{algebraic-families-Di-equation})$ that
\begin{align*}
t^2 = p_0\left(p_0\epsilon - 4b_z\text{sign}(a_z)\left(\prod_{w \in \Z\setminus \{z\}}\max(1, |a_w|)\right)\left(\prod_{l \in \S, l \ne 2}l \right)
\left(\prod_{w \in \Z\setminus \{z\}}(D_w - 1)\right)\right).
\end{align*}
This implies that $t \in \bZ$ and $t \equiv 0 \pmod{p_0}$. Hence $v_{p_0}(t^2) = 2v_{p_0}(t) \ge 2$. Thus we deduce that
\begin{align*}
p_0\epsilon - 4b_z\text{sign}(a_z)\left(\prod_{w \in \Z\setminus \{z\}}\max(1, |a_w|)\right)\left(\prod_{l \in \S, l \ne 2}l \right)
\left(\prod_{w \in \Z\setminus \{z\}}(D_w - 1)\right) \equiv 0 \pmod{p_0}.
\end{align*}
Hence
\begin{align*}
4b_z\text{sign}(a_z)\left(\prod_{w \in \Z\setminus \{z\}}\max(1, |a_w|)\right)\left(\prod_{l \in \S, l \ne 2}l \right)
\left(\prod_{w \in \Z\setminus \{z\}}(D_w - 1)\right)  \equiv 0 \pmod{p_0}.
\end{align*}
By $(\ref{algebraic-families-Di-equation})$, one knows that $D_w \equiv 0 \pmod{p_0}$ for every $w \in \Z$. Hence
\begin{align*}
\prod_{w \in \Z\setminus \{z\}}(D_w - 1) \equiv (-1)^{m - 1} \pmod{p_0}.
\end{align*}
Thus we deduce that
\begin{align*}
(-1)^{m - 1}4b_z\text{sign}(a_z)\left(\prod_{w \in \Z\setminus \{z\}}\max(1, |a_w|)\right)\left(\prod_{l \in \S, l \ne 2}l \right)  \equiv 0 \pmod{p_0},
\end{align*}
which is a contradiction to the choice of $p_0$. Therefore $\F(t) \ne z$ for every $t \in \bQ$.

$\bullet$ \textit{Subcase 2. $a_z = 0$.}

We see that $\F(t) = \dfrac{a_z}{b_z} = 0$. Hence we deduce from the definition of $\F(t)$ that
\begin{align*}
t^2 = p_0^2\epsilon - 4\left(\prod_{l \in \S, l \ne 2}l \right)\left(\prod_{w \in \Z}(D_w - 1)\right).
\end{align*}
This implies that $t \in \bZ$. Hence we deduce that
\begin{align*}
t^2 = p_0^2\epsilon \pmod{l}
\end{align*}
for each prime $l \in \S$ with $l \ne 2$. Since $\epsilon$ is a quadratic non-residue in $\bF_l^{\times}$, it follows that $t \equiv p_0 \equiv 0 \pmod{l}$, which is a contradiction to the choice of $p_0$. Thus, in any event, $\F(t) \ne z$ for every $t \in \bQ$. Therefore $\F(t)$ satisfies $(2)$ in Lemma \ref{algebraic-family-hyperelliptic-curve-lemma}.

$\star$ \textit{Case 2. $\Z = \emptyset$.}

In this case, let $\epsilon$ be the same as in \textit{Case 1}, and let $p_0$ be an odd prime such that $p_0 \not\in \S$. Let $\F(t) \in \bQ(t)$ be the rational function defined by
\begin{equation}
\label{rational-function-F(t)-empty-case}
\F(t) := 1 + \dfrac{4\left(\prod_{l \in \S, l \ne 2}l \right)}{t^2 - p_0^2\epsilon}.
\end{equation}
Using the same arguments as in \textit{Case 1}, one can show that $\F(t)$ satisfies $(1)$ and $(2)$ in Lemma \ref{algebraic-family-hyperelliptic-curve-lemma}.

\end{proof}

\begin{lemma}
\label{Lemma-Rational-functions-of-degree-4-belong-to-q-Z-q}

Let $\D(t) \in \bQ(t)$ be a nonzero rational function of the form
\begin{align*}
\D(t) := \dfrac{at^4 + bt^2 + c}{dt^4 + et^2 + f},
\end{align*}
where $a, b, c, d, e, f$ are integers. Let $q$ be an odd prime. Assume that there exists an integer $t_0$ such that
\begin{align*}
at_0^4 + bt_0^2 + c \equiv 0 \pmod{q},
\end{align*}
\begin{align*}
at_0^4 + bt_0^2 + c \not\equiv 0 \pmod{q^2},
\end{align*}
and
\begin{align*}
dt_0^4 + et_0^2 + f \not\equiv 0 \pmod{q}.
\end{align*}
Then there exists a rational function $\Gamma(t) \in \bQ(t)$ such that for every $t \in \bQ$, $\D(\Gamma(t))$ belongs to $q\bZ_q$, but does not belong to $q^2\bZ_q$.

\end{lemma}

\begin{proof}

Let $\epsilon$ be an integer such that $\epsilon$ is a quadratic non-residue in $\bF_q^{\times}$. Let $q_0$ be an odd prime such that $q_0 \ne q$. We will show that the rational function $\Gamma(t) \in \bQ(t)$ defined by
\begin{align}
\label{Definition-The-definition-of-Gamma-function}
\Gamma(t) := t_0 + \dfrac{q^2}{t^2 - q_0^2\epsilon}
\end{align}
satisfies the assertions in Lemma \ref{Lemma-Rational-functions-of-degree-4-belong-to-q-Z-q}.

Since $\epsilon$ is not a square in $\bF_q^{\times}$, it follows that $t^2 - q_0^2\epsilon$ is non-zero for each $t \in \bQ$, and hence $\Gamma(t)$ is well-defined for every $t \in \bQ$.

We now prove that $\Gamma(t)$ belongs to $t_0 + q^2\bZ_q$ for each $t \in \bQ$. Indeed, take any rational number $t$, and write $t = \dfrac{t_1}{t_2}$, where $t_1, t_2$ are integers such that $t_2 \ne 0$ and $\gcd(t_1, t_2) = 1$. We see that
\begin{align*}
v_q\left(\dfrac{1}{t^2 - q_0^2\epsilon}\right) = v_q\left(\dfrac{t_2^2}{t_1^2 - q_0^2\epsilon t_2^2}\right) = v_q(t_2^2) - v_q(t_1^2 - q_0^2\epsilon t_2^2).
\end{align*}

If $t_2 \equiv 0 \pmod{q}$, then it follows that $t_1 \not\equiv 0 \pmod{q}$. Hence we deduce that
\begin{align*}
v_q(t_1^2 - q_0^2\epsilon t_2^2) = \min\left(v_q(t_1^2), v_q(q_0^2\epsilon t_2^2)\right) = \min\left(0, v_q(t_2^2)\right) = 0,
\end{align*}
and thus
\begin{align*}
v_q\left(\dfrac{1}{t^2 - q_0^2\epsilon}\right) = v_q(t_2^2) - v_q(t_1^2 - q_0^2\epsilon t_2^2) = 2v_q(t_2) \ge 2.
\end{align*}
Therefore $\dfrac{1}{t^2 - q_0^2\epsilon}$ belongs to $\bZ_q$, and hence it follows from $(\ref{Definition-The-definition-of-Gamma-function})$ that $\Gamma(t)$ belongs to $t_0 + q^2\bZ_q$.

If $t_2 \not\equiv 0 \pmod{q}$, it follows that $v_q(t_2^2) = 0$. We contend that $t_1^2 - q_0^2\epsilon t_2^2 \not\equiv 0 \pmod{q}$. Assume the contrary, that is, $t_1^2 - q_0^2\epsilon t_2^2 \equiv 0 \pmod{q}$. Since $t_2 \not\equiv 0 \pmod{q}$ and $q_0 \ne q$, we deduce that
\begin{align*}
\epsilon \equiv \left(\dfrac{t_1}{q_0t_2}\right)^2 \pmod{q},
\end{align*}
which contradicts the choice of $\epsilon$. This contradiction establishes that $t_1^2 - q_0^2\epsilon t_2^2 \not\equiv 0 \pmod{q}$, and thus
\begin{align*}
v_q\left(\dfrac{1}{t^2 - q_0^2\epsilon}\right) = v_q(t_2^2) - v_q(t_1^2 - q_0^2\epsilon t_2^2) = 0.
\end{align*}
Therefore $\dfrac{1}{t^2 - q_0^2\epsilon}$ belongs to $\bZ_q^{\times}$, and hence it follows from $(\ref{Definition-The-definition-of-Gamma-function})$ that $\Gamma(t)$ belongs to $t_0 + q^2\bZ_q$.

Since $\Gamma(t)$ belongs to $t_0 + q^2\bZ_q$, we see that
\begin{align*}
a(\Gamma(t))^4 + b(\Gamma(t))^2 + c \equiv at_0^4 + bt_0^2 + c \equiv 0 \pmod{q},
\end{align*}
\begin{align*}
a(\Gamma(t))^4 + b(\Gamma(t))^2 + c \equiv at_0^4 + bt_0^2 + c \not\equiv 0 \pmod{q^2},
\end{align*}
and
\begin{align*}
d(\Gamma(t))^4 + e(\Gamma(t))^2 + f \equiv dt_0^4 + et_0^2 + f \not\equiv 0 \pmod{q}
\end{align*}
for each $t \in \bQ$. The last congruence shows that
\begin{align*}
\dfrac{1}{c(\Gamma(t))^4 + d(\Gamma(t))^2 + e}
\end{align*}
belongs to $\bZ_q^{\times}$, and hence we deduce that for every $t \in \bQ$,
\begin{align*}
\D(\Gamma(t)) = \dfrac{a(\Gamma(t))^4 + b(\Gamma(t))^2 + c}{d(\Gamma(t))^4 + e(\Gamma(t))^2 + f}
\end{align*}
belongs to $q\bZ_q$, but does not belong to $q^2\bZ_q$. Thus our contention follows.

\end{proof}

The following two examples will be used in the proof of the main theorem in this section.

\begin{example}
\label{Example-The-first-rational-function-D1-related-to-the-construction-of-algebraic-families}

Let $\D_1(T) \in \bQ(T)$ be the rational function defined by
\begin{align}
\label{Definition-The-rational-function-D1-T}
\D_1(T) := \frac{45588894173298T^4 -  1641200890885920T^2 + 14770814323798008}{-5477180725633679T^4 + 197178506122812676T^2 - 1774606555105302716},
\end{align}
and define
\begin{align}
\label{Definition-The-rational-function-D1-ast-T}
\D_1^{\ast}(T) : = 7 + 1682\D_1(T) = \frac{-38340254920051483T^4  + 1380250355610428708T^2  -12422263806891130444}{5477180725633679T^4 - 197178506122812676T^2 + 1774606555105302716}.
\end{align}

Let $T_0 = 0$, and let $q = 31$. Since
\begin{align*}
12422263806891130444 = 2^2\cdot7^3\cdot 31 \cdot 433\cdot 3299 \cdot 10589 \cdot 19309,
\end{align*}
it follows that
\begin{align*}
v_q(-12422263806891130444) = v_{31}(-12422263806891130444) = 1,
\end{align*}
and hence we deduce that
\begin{align*}
&-38340254920051483T_0^4  + 1380250355610428708T_0^2 - 12422263806891130444 \\
&= -12422263806891130444 \equiv 0 \pmod{q}
\end{align*}
and
\begin{align*}
&-38340254920051483T_0^4  + 1380250355610428708T_0^2 - 12422263806891130444 \\
&= -12422263806891130444 \not\equiv 0 \pmod{q^2}.
\end{align*}
Since
\begin{align*}
1774606555105302716 = 2^2 \cdot 7^2 \cdot 47 \cdot 192640746320593,
\end{align*}
we see that
\begin{align*}
&5477180725633679T_0^4 - 197178506122812676T_0^2 + 1774606555105302716 \\
&= 1774606555105302716 \not\equiv 0 \pmod{q}.
\end{align*}

Let $\epsilon = 3$, and let $q_0 = 5$. Following the proof of Lemma \ref{Lemma-Rational-functions-of-degree-4-belong-to-q-Z-q}, we define the rational function $\Gamma_1(T) \in \bQ(T)$ by $(\ref{Definition-The-definition-of-Gamma-function})$, that is,
\begin{align}
\label{Definition-The-definition-of-the-rational-function-Gamma1}
\Gamma_1(T) := T_0 + \dfrac{q^2}{T^2 - q_0^2\epsilon} = \dfrac{961}{T^2 - 75}.
\end{align}
Applying Lemma \ref{Lemma-Rational-functions-of-degree-4-belong-to-q-Z-q} with $\D_1^{\ast}(T)$ in the role of $\D(t)$, we deduce that the rational function $\D_1^{\ast}(\Gamma_1(T))$ belongs to $31\bZ_{31}$, but does not belong to $31^2\bZ_{31}$, where
\begin{align}
\label{Definition-The-rational-function-D-1-ast-of-Gamma-1}
\D_1^{\ast}(\Gamma_1(T)) = \frac{\Sigma_{1, 1}(T)}{\Sigma_{1, 2}(T)}
\end{align}
with
\begin{align}
\label{Equation-Sigma-1-1}
\Sigma_{1, 1}(T) &= -12422263806891130444 T^8 + 3726679142067339133200 T^6 \\
&+ 855438785181123078355868 T^4 - 170240958125426027001880200 T^2 \nonumber \\
&- 25922975674046723162225380003 \nonumber
\end{align}
and
\begin{align}
\label{Equation-Sigma-1-2}
\Sigma_{1, 2}(T) &= 1774606555105302716 T^8 - 532381966531590814800 T^6 \\
&- 122205519918242118687196 T^4 + 24320125111216714469579400 T^2 \nonumber \\
&+ 3703283999134302153081910439. \nonumber
\end{align}

\end{example}

\begin{example}
\label{Example-The-second-rational-function-D2-related-to-the-construction-of-algebraic-families}

Let $\D_2(T) \in \bQ(T)$ be the rational function defined by
\begin{align}
\label{Definition-The-rational-function-D2-T}
\D_2(T) := \frac{-64380401708754T^4 + 2317693623118880T^2 - 20859235062503544}{407097080892401T^4 - 14655494912126204T^2 + 131899454209147204},
\end{align}
and define
\begin{align}
\label{Definition-The-rational-function-D2-ast-T}
\D_2^{\ast}(T) : = 133 + 1682\D_2(T) = \frac{-54143923915434895T^4  + 1949179850773171028T^2  - 17542605965314382876}{407097080892401T^4 - 14655494912126204T^2 + 131899454209147204}.
\end{align}

Let $T_0 = 0$, and let $q = 11$. Since
\begin{align*}
17542605965314382876 = 2^2 \cdot 7 \cdot 11 \cdot 56956512874397347,
\end{align*}
it follows that
\begin{align*}
-54143923915434895T_0^4  + 1949179850773171028T_0^2  - 17542605965314382876 \equiv 0 \pmod{11}
\end{align*}
but
\begin{align*}
-54143923915434895T_0^4  + 1949179850773171028T_0^2  - 17542605965314382876 \not\equiv 0 \pmod{11^2}.
\end{align*}

Since
\begin{align*}
131899454209147204 \equiv 8 \not\equiv 0 \pmod{11},
\end{align*}
we deduce that
\begin{align*}
407097080892401T^4 - 14655494912126204T^2 + 131899454209147204 \not\equiv 0 \pmod{7}.
\end{align*}

Let $\epsilon = 7$, and let $q_0 = 3$. Following the proof of Lemma \ref{Lemma-Rational-functions-of-degree-4-belong-to-q-Z-q}, we define the rational function $\Gamma_2(T) \in \bQ(T)$ by $(\ref{Definition-The-definition-of-Gamma-function})$, that is,
\begin{align}
\label{Definition-The-definition-of-the-rational-function-Gamma2}
\Gamma_2(T) := T_0 + \dfrac{q^2}{T^2 - q_0^2\epsilon} = \dfrac{121}{T^2 - 63}.
\end{align}
Applying Lemma \ref{Lemma-Rational-functions-of-degree-4-belong-to-q-Z-q} with $\D_2^{\ast}(T)$ in the role of $\D(t)$, we deduce that the rational function $\D_2^{\ast}(\Gamma_2(T))$ belongs to $11\bZ_{11}$, but does not belong to $11^2\bZ_{11}$, where
\begin{align}
\label{Definition-The-rational-function-D-2-ast-of-Gamma-2}
\D_2^{\ast}(\Gamma_2(T)) = \frac{\Sigma_{2, 1}(T)}{\Sigma_{2, 2}(T)}
\end{align}
with
\begin{align}
\label{Equation-Sigma-2-1}
\Sigma_{2, 1}(T) &= -17542605965314382876 T^8 + 4420736703259224484752 T^6 \\
&-389221676262826716788116 T^4 + 13950123258644442355341240 T^2 \nonumber \\
&-174687125980796870729105719 \nonumber
\end{align}
and
\begin{align}
\label{Equation-Sigma-2-2}
\Sigma_{2, 2}(T) &= 131899454209147204 T^8 - 33238662460705095408 T^6 \\
&+ 2926482501528191763292 T^4  - 104888292579475114826088 T^2 \nonumber \\
&+ 1313439132893945928914009. \nonumber
\end{align}

\end{example}

We recall the following separability criterion \cite{dq-separable-polynomials}.

\begin{theorem}
\label{the-criterion-for-separability-of-certain-polynomials-theorem}

Let $n, m, k$ be positive integers, and let $a, b, c, d, e$ be rational numbers such that $a \ne 0$. Let $p$ be an odd prime such that $a, b, c, d, e$ belong to $\bZ_p$ and $a \equiv 0 \pmod{p}$. Let $F(x) \in \bQ[x]$ be the polynomial defined by
\begin{align}
\label{the-polynomial-F-equation}
F(x) := ax^{2n + 2} + (bx^{2m} + c)(dx^{2k} + e).
\end{align}
Define
\begin{align*}
n_1 &:= (m + k)(v_p(a) - v_p(bd)) + m + k - 1, \\
n_2 &:= (m + k)(v_p(a) - v_p(b)) + m - 1, \\
n_3 &:= (m + k)(v_p(a) - v_p(d)) + k - 1, \\
n_4 &:= (m + k)v_p(a) - 1, \\
n_5 &:= v_p(a) - v_p(bd) + m + k - 1.
\end{align*}
Suppose that the following are true:
\begin{itemize}

\item [(S1)] $n > m + k - 1$ and $n > \max(n_1, n_2, n_3, n_4, n_5)$.

\item [(S2)] $ce \not\equiv 0 \pmod{p}$, $km \not\equiv 0 \pmod{p}$, and $b^ke^m + (-1)^{m + k + 1}c^kd^m \not\equiv 0 \pmod{p}$.

\end{itemize}
Then $F$ is separable: that is, it has exactly $2n + 2$ distinct roots in $\bC$.

\end{theorem}

\begin{remark}

Theorem \ref{the-criterion-for-separability-of-certain-polynomials-theorem} is a very mild generalization of Theorem 2.1 in \cite{dq-separable-polynomials}. The only difference between Theorem \ref{the-criterion-for-separability-of-certain-polynomials-theorem} and \cite[Theorem 2.1]{dq-separable-polynomials} is that in \cite[Theorem 2.1]{dq-separable-polynomials}, the author assumed that $a, b, c, d, e$ are integers, whereas in Theorem \ref{the-criterion-for-separability-of-certain-polynomials-theorem} presented here, we only need the assumption that $a, b, c, d, e$ belong to $\bZ_p$. Upon looking closely the proof of \cite[Theorem 2.1]{dq-separable-polynomials}, we see that it is sufficient to assume that $a, b, c, d, e$ are in $\bZ_p$, and hence Theorem \ref{the-criterion-for-separability-of-certain-polynomials-theorem} follows immediately from \cite[Theorem 2.1]{dq-separable-polynomials}.

\end{remark}

Using Theorem \ref{the-criterion-for-separability-of-certain-polynomials-theorem}, we prove the following corollaries that play a key role in constructing algebraic families of curves violating the Hasse principle.

\begin{corollary}
\label{Corollary-The-separability-of-the-first-algebraic-family-of-hyperelliptic-curves}

We maintain the same notation as in Example \ref{Example-The-first-rational-function-D1-related-to-the-construction-of-algebraic-families}. Let $\D_1(T)$, $\D_1^{\ast}(T), \Gamma_1(T) \in \bQ(T)$ be the rational functions defined by $(\ref{Definition-The-rational-function-D1-T})$, $(\ref{Definition-The-rational-function-D1-ast-T})$, $(\ref{Definition-The-definition-of-the-rational-function-Gamma1})$, respectively in Example \ref{Example-The-first-rational-function-D1-related-to-the-construction-of-algebraic-families}. Let $n$ be a positive integer such that $n > 5$. For each rational number $T \in \bQ$, let $\cP_{1, T}(x) \in \bQ[x]$ be the polynomial of degree $2n + 2$ given by
\begin{align}
\label{Definition-The-definition-of-the-polynomial-P-1-T}
&\cP_{1, T}(x) := 118579927725(\D_1^{\ast}(\Gamma_1(T)))^2x^{2n + 2}  \\
&+ \left(2\left(29(\D_1^{\ast}(\Gamma_1(T)))^2 + 123192\right)x^2 - 16689645\right)\left(261\left(29(\D_1^{\ast}(\Gamma_1(T)))^2 + 123192\right)x^2 - 33379290\right), \nonumber
\end{align}
where the composition $\D_1^{\ast}(\Gamma_1(T))$ of $\D_1^{\ast}(T)$ and $\Gamma_1(T)$ is given by $(\ref{Definition-The-rational-function-D-1-ast-of-Gamma-1})$ in Example \ref{Example-The-first-rational-function-D1-related-to-the-construction-of-algebraic-families}. Then for each $T \in \bQ$, the polynomial $\cP_{1, T}(x)$ is separable: that is, it has exactly $2n + 2$ distinct roots in $\bC$.

\end{corollary}

\begin{proof}

Throughout the proof, we maintain the same notation as in Theorem \ref{the-criterion-for-separability-of-certain-polynomials-theorem}. Take any rational number $T \in \bQ$, and define
\begin{align*}
a &:= 118579927725(\D_1^{\ast}(\Gamma_1(T)))^2, \\
b &:= 2\left(29(\D_1^{\ast}(\Gamma_1(T)))^2 + 123192\right), \\
c &:= -16689645, \\
d &:= 261\left(29(\D_1^{\ast}(\Gamma_1(T)))^2 + 123192\right), \\
e &:= -33379290.
\end{align*}
Let $p = 31$, and let $m = k = 1$. Since $118579927725 \equiv 27 \not\equiv 0 \pmod{31}$, $123192 \equiv 29 \not\equiv 0 \pmod{31}$, it follows from Example \ref{Example-The-first-rational-function-D1-related-to-the-construction-of-algebraic-families} that
\begin{align*}
v_{p}(a) = v_{31}\left( 118579927725(\D_1^{\ast}(\Gamma_1(T)))^2 \right) = 2v_{31}\left(\D_1^{\ast}(\Gamma_1(T)) \right) = 2,
\end{align*}
\begin{align*}
v_p(b) = v_{31}\left(2\left(29(\D_1^{\ast}(\Gamma_1(T)))^2 + 123192\right)\right) = v_{31}(123192) = 0,
\end{align*}
and
\begin{align*}
v_p(d) = v_{31}\left(261\left(29(\D_1^{\ast}(\Gamma_1(T)))^2 + 123192\right)\right) = v_{31}(123192) = 0.
\end{align*}

We see that
\begin{align*}
n_1 &:= (m + k)(v_p(a) - v_p(bd)) + m + k - 1 = 2v_p(a) + 1 = 5, \\
n_2 &:= (m + k)(v_p(a) - v_p(b)) + m - 1 = 2v_p(a) = 4, \\
n_3 &:= (m + k)(v_p(a) - v_p(d)) + k - 1 = 2v_p(a) = 4, \\
n_4 &:= (m + k)v_p(a) - 1 = 2v_p(a) - 1 = 3, \\
n_5 &:= v_p(a) - v_p(bd) + m + k - 1 = 2 + 1 = 3,
\end{align*}
and hence
\begin{align*}
\max(n_1, n_2, n_3, n_4, n_5) = 5.
\end{align*}
By assumption, we know that
\begin{align*}
n > 5 = \max(n_1, n_2, n_3, n_4, n_5),
\end{align*}
and hence condition $(S1)$ in Theorem \ref{the-criterion-for-separability-of-certain-polynomials-theorem} is satisfied.

It is obvious that $km = 1 \not\equiv 0 \pmod{31}$ and
\begin{align*}
ce = (-16689645)\cdot (-33379290) \equiv 25 \not\equiv 0 \pmod{31}.
\end{align*}
Furthermore, since $\D_1^{\ast}(\Gamma_1(T))$ belongs to $31\bZ_{31}$, we deduce that
\begin{align*}
b^ke^m + (-1)^{m + k + 1}c^kd^m &= be - cd \\
&= \left(2\left(29(\D_1^{\ast}(\Gamma_1(T)))^2 + 123192\right)\right)(-33379290) \\
&- (-16689645)\left(261\left(29(\D_1^{\ast}(\Gamma_1(T)))^2 + 123192\right)\right) \\
&\equiv 2\cdot 123192 \cdot (-33379290) + 16689645 \cdot 261 \cdot 123192 \\
&\equiv 12 \not\equiv 0 \pmod{31}.
\end{align*}
Therefore condition $(S2)$ in Theorem \ref{the-criterion-for-separability-of-certain-polynomials-theorem} is satisfied, and hence the polynomial $\cP_{1, T}(x)$ is separable. Since $T$ is an arbitrary rational number, our contention follows.

\end{proof}

\begin{corollary}
\label{Corollary-The-separability-of-the-second-algebraic-family-of-hyperelliptic-curves}

We maintain the same notation as in Example \ref{Example-The-second-rational-function-D2-related-to-the-construction-of-algebraic-families}. Let $\D_2(T)$, $\D_2^{\ast}(T), \Gamma_2(T) \in \bQ(T)$ be the rational functions defined by $(\ref{Definition-The-rational-function-D2-T})$, $(\ref{Definition-The-rational-function-D2-ast-T})$, $(\ref{Definition-The-definition-of-the-rational-function-Gamma2})$, respectively in Example \ref{Example-The-second-rational-function-D2-related-to-the-construction-of-algebraic-families}. Let $n$ be a positive integer such that $n > 5$. For each rational number $T \in \bQ$, let $\cP_{2, T}(x) \in \bQ[x]$ be the polynomial of degree $2n + 2$ given by
\begin{align}
\label{Definition-The-definition-of-the-polynomial-P-2-T}
&\cP_{2, T}(x) := 84898109(\D_2^{\ast}(\Gamma_2(T)))^2 x^{2n + 2} \\
&+ \left(2\left(29(\D_2^{\ast}(\Gamma_2(T)))^2 - 40600\right)x^2 + 49619\right)\left(261\left(29(\D_2^{\ast}(\Gamma_2(T)))^2 - 40600\right)x^2 + 99238\right),
\end{align}
where the composition $\D_2^{\ast}(\Gamma_2(T))$ of $\D_2^{\ast}(T)$ and $\Gamma_2(T)$ is given by $(\ref{Definition-The-rational-function-D-2-ast-of-Gamma-2})$ in Example \ref{Example-The-second-rational-function-D2-related-to-the-construction-of-algebraic-families}. Then for each $T \in \bQ$, the polynomial $\cP_{2, T}(x)$ is separable: that is, it has exactly $2n + 2$ distinct roots in $\bC$.

\end{corollary}

\begin{proof}

Throughout the proof, we maintain the same notation as in Theorem \ref{the-criterion-for-separability-of-certain-polynomials-theorem}. Take any rational number $T \in \bQ$, and define
\begin{align*}
a &:= 84898109(\D_2^{\ast}(\Gamma_2(T)))^2, \\
b &:= 2\left(29(\D_2^{\ast}(\Gamma_2(T)))^2 - 40600\right), \\
c &:= 49619, \\
d &:= 261\left(29(\D_2^{\ast}(\Gamma_2(T)))^2 - 40600\right), \\
e &:= 99238.
\end{align*}
Let $p = 11$, and let $m = k = 1$. Since $84898109 \equiv 10 \not\equiv 0 \pmod{11}$, $40600 \equiv 10 \not\equiv 0 \pmod{11}$, it follows from Example \ref{Example-The-second-rational-function-D2-related-to-the-construction-of-algebraic-families} that
\begin{align*}
v_{p}(a) = v_{11}\left(84898109(\D_2^{\ast}(\Gamma_2(T)))^2 \right) = 2v_{11}\left(\D_2^{\ast}(\Gamma_2(T))\right) = 2,
\end{align*}
\begin{align*}
v_p(b) = v_{11}\left(2\left(29(\D_2^{\ast}(\Gamma_2(T)))^2 - 40600\right)\right) = v_{11}(40600) = 0,
\end{align*}
and
\begin{align*}
v_p(d) = v_{31}\left(261\left(29(\D_2^{\ast}(\Gamma_2(T)))^2 - 40600\right)\right) = v_{11}(40600) = 0.
\end{align*}

We see that
\begin{align*}
n_1 &:= (m + k)(v_p(a) - v_p(bd)) + m + k - 1 = 2v_p(a) + 1 = 5, \\
n_2 &:= (m + k)(v_p(a) - v_p(b)) + m - 1 = 2v_p(a) = 4, \\
n_3 &:= (m + k)(v_p(a) - v_p(d)) + k - 1 = 2v_p(a) = 4, \\
n_4 &:= (m + k)v_p(a) - 1 = 2v_p(a) - 1 = 3, \\
n_5 &:= v_p(a) - v_p(bd) + m + k - 1 = 2 + 1 = 3,
\end{align*}
and hence
\begin{align*}
\max(n_1, n_2, n_3, n_4, n_5) = 5.
\end{align*}
By assumption, we know that
\begin{align*}
n > 5 = \max(n_1, n_2, n_3, n_4, n_5),
\end{align*}
and hence condition $(S1)$ in Theorem \ref{the-criterion-for-separability-of-certain-polynomials-theorem} is satisfied.

It is obvious that $km = 1 \not\equiv 0 \pmod{11}$ and
\begin{align*}
ce = 49619 \cdot 99238 \equiv 8 \not\equiv 0 \pmod{11}.
\end{align*}
Since $\D_2^{\ast}(\Gamma_2(T))$ belongs to $11\bZ_{11}$, we deduce that
\begin{align*}
b^ke^m + (-1)^{m + k + 1}c^kd^m &= be - cd \\
&= \left(2\left(29(\D_2^{\ast}(\Gamma_2(T)))^2 - 40600\right)\right)(99238) \\
&- (49619)\left(261\left(29(\D_2^{\ast}(\Gamma_2(T)))^2 - 40600\right)\right) \\
&\equiv 2\cdot (-40600) \cdot 99238 - (49619) \cdot 261 \cdot (-40600) \\
&\equiv 8 \not\equiv 0 \pmod{11}.
\end{align*}
Therefore condition $(S2)$ in Theorem \ref{the-criterion-for-separability-of-certain-polynomials-theorem} is satisfied, and hence the polynomial $\cP_{2, T}(x)$ is separable. Since $T$ is an arbitrary rational number, our contention follows.

\end{proof}

For the rest of this section, let
\begin{align*}
\A_1 := \left\{n \in \bZ: n > 5, \; \; n \not\equiv 0 \pmod{4} \; \; \text{and} \; \; n \not\equiv 21 \pmod{29} \right\},
\end{align*}
and
\begin{align*}
\A_2 := \left\{n \in \bZ: n > 5, \; \; n \not\equiv 0 \pmod{4} \; \; \text{and} \; \; n \not\equiv 8 \pmod{29} \right\}.
\end{align*}
We see that
\begin{align}
\label{Equation-the-union-of-A1-and-A2-is-the-set-of-all-n-greater-than-5-and-not-equiv-to-zero-mod-4}
\A_1 \cup \A_2 = \left\{n \in \bZ: n > 5 \; \; \text{and} \; \; n \not\equiv 0 \pmod{4} \right\}.
\end{align}

We now prove the main theorem in this section.

\begin{theorem}
\label{algebraic-family-hyperelliptic-curve-hp-theorem}

For each $n \in \A_1$ and each rational number $T \in \bQ$, let $\cP_{1, T}(x) \in \bQ[x]$ be the polynomial of degree $2n + 2$ defined by $(\ref{Definition-The-definition-of-the-polynomial-P-1-T})$ in Corollary \ref{Corollary-The-separability-of-the-first-algebraic-family-of-hyperelliptic-curves}. For each $n \in \A_2$ and each rational number $T \in \bQ$, let $\cP_{2, T}(x) \in \bQ[x]$ be the polynomial of degree $2n + 2$ defined by $(\ref{Definition-The-definition-of-the-polynomial-P-2-T})$ in Corollary \ref{Corollary-The-separability-of-the-second-algebraic-family-of-hyperelliptic-curves}. Then
\begin{itemize}

\item [(i)] For each $n \in \A_1$ and each rational number $T \in \bQ$, the hyperelliptic curve $\cC_{n, T, (29, 1, 3)}^{(7, 261, 15)}$ of genus $n$ satisfies HP1 and HP2, where $\cC_{n, T, (29, 1, 3)}^{(7, 261, 15)}$ is the smooth projective model of the affine curve defined by
\begin{align*}
\cC_{n, T, (29, 1, 3)}^{(7, 261, 15)} : z^2 = \cP_{1, T}(x).
\end{align*}

\item [(ii)] For each $n \in \A_2$ and each rational number $T \in \bQ$, the hyperelliptic curve $\cC_{n, T, (29, 1, 3)}^{(133, 29, 27)}$ of genus $n$ satisfies HP1 and HP2, where $\cC_{n, T, (29, 1, 3)}^{(133, 29, 27)}$ is the smooth projective model of the affine curve defined by
\begin{align*}
\cC_{n, T, (29, 1, 3)}^{(133, 29, 27)} &: z^2 = \cP_{2, T}(x).
\end{align*}

\end{itemize}

\end{theorem}

\begin{remark}
\label{Remark-The-above-theorem-implies-the-main-theorem-in-the-paper}

By $(\ref{Equation-the-union-of-A1-and-A2-is-the-set-of-all-n-greater-than-5-and-not-equiv-to-zero-mod-4})$ and Theorem \ref{algebraic-family-hyperelliptic-curve-hp-theorem}, we see that Theorem \ref{Theorem-The-main-theorem} follows immediately.

\end{remark}

\begin{proof}

Throughout the proof of Theorem \ref{algebraic-family-hyperelliptic-curve-hp-theorem}, we will use the same notation as in Theorem \ref{hyperelliptic-curves-theorem} and Lemma \ref{infinitude-alpha-beta-gamma-lemma}. We first prove that part $(i)$ of Theorem \ref{algebraic-family-hyperelliptic-curve-hp-theorem} holds.

Let $n$ be any integer such that $n \in \A_1$. Let $(p, b, d, \alpha_0, \beta_0, \gamma_0) = (29, 1, 3, 7, 261, 15)$. We see that $\bar{\beta_0} = 9$. One can check that the sextuple $(p, b, d, \alpha_0, \beta_0, \gamma_0)$ satisfies $(A1)-(A5)$ and $(B1)$. Indeed, $(A1), (A2), (A4)$, $(A5)$ and $(B1)$ are obvious. It remains to prove that $(p, b, d, \alpha_0, \beta_0, \gamma_0)$ satisfies $(A3)$. By $(\ref{eqn-P})$ and $(\ref{eqn-Q})$, we know that
\begin{align*}
P_0 &= 124613, \\
Q_0 &= -63945.
\end{align*}
The conic $\cQ_1^{(7, 261, 15)}$ in $(A3)$ of Theorem \ref{hyperelliptic-curves-theorem} defined by
\begin{equation*}
\cQ_1^{(7, 261, 15)} : 29U^2 - V^2 + 2079746732385T^2 = 0
\end{equation*}
has a point $(u_0, v_0, t_0) = (166257, 3020031, 2)$, and hence $(p, b, d, \alpha_0, \beta_0, \gamma_0)$ satisfies $(A3)$.

Let $\S := \{2, 29\}$, and let $\C_1(T)$ be the rational function in $\bQ(T)$ defined by the same equation $(\ref{eqn-C})$ of $C$ with $(0, T)$ in the role of $(A, B)$, that is,
\begin{align*}
\C_1(T) := \dfrac{- 2v_0T - 4p^3\alpha_0\beta_0t_0^2Q_0}{T^2 + 4p^5\beta_0t_0^2Q_0} = \dfrac{-6040062T + 45588900213360}{T^2 - 5477180725633680}.
\end{align*}
Let $\G_1(T) \in \bQ(T)$ be the rational function defined by
\begin{align}
\label{G1(T)-rational-function-eqn}
\G_1(T) = v_0 + T\C_1(T) = \frac{-3020031T^2 + 45588900213360T - 16541255584016208244080}{T^2 - 5477180725633680}.
\end{align}
Since the numerator and denominator of $\G_1(T)$ are irreducible polynomials over $\bQ$, the set $\Z_1$ of rational zeros and poles of $\G_1(T)$ is empty. Hence applying Lemma \ref{algebraic-family-hyperelliptic-curve-lemma} for the triple $\left(\S, \G_1(T), \Z_1 \right)$, we know that $\F_1(T)$ satisfies $(1)$ and $(2)$ in Lemma \ref{algebraic-family-hyperelliptic-curve-lemma}, where $\F_1(T)$ is the rational function defined by $(\ref{rational-function-F(t)-empty-case})$ in \textit{Case 2} of Lemma \ref{algebraic-family-hyperelliptic-curve-lemma} with $(p_0, \epsilon) = (3, 2)$ and $\left(\S, \G_1(T), \Z_1 \right)$ in the role of $(\S, \G(T), \Z)$, that is,
\begin{align*}
\F_1(T) := 1 + \dfrac{4\left(\prod_{l \in \S, l \ne 2}l \right)}{T^2 - p_0^2\epsilon} = 1 + \dfrac{116}{T^2 - 18} = \dfrac{T^2 + 98}{T^2 - 18}.
\end{align*}

Let $\Gamma_1(T) \in \bQ(T)$ be the rational function defined by $(\ref{Definition-The-definition-of-the-rational-function-Gamma1})$ in Example \ref{Example-The-first-rational-function-D1-related-to-the-construction-of-algebraic-families}. Recall that
\begin{align*}
\Gamma_1(T) := \dfrac{961}{T^2 - 75}.
\end{align*}
It is known that $\Gamma_1(T)$ is well-defined for each rational number $T \in \bQ$.

Let $(A, B) = (0, \F_1(\Gamma_1(T)))$. By Lemma \ref{algebraic-family-hyperelliptic-curve-lemma}, we know that $(0, \F_1(T))$ satisfies $(C1)$ and $(C2)$ in Lemma \ref{infinitude-alpha-beta-gamma-lemma} for each $T \in \bQ$. Thus it follows that $(A, B) = (0, \F_1(\Gamma_1(T)))$ satisfies $(C1)$ and $(C2)$ in Lemma \ref{infinitude-alpha-beta-gamma-lemma} for each $T \in \bQ$.

Let $\D_1(\Gamma_1(T))$ be the rational function in $\bQ(T)$ defined by the same equation $(\ref{eqn-C})$ of $C$ with $(0, \F_1(\Gamma_1(T)))$ in the role of $(A, B)$, that is,
\begin{align}
\label{D1-eqn-algebraic-family}
&\D_1(\Gamma_1(T)) := \C_1(\F_1(\Gamma_1(T))) \\
&= \frac{45588894173298(\Gamma_1(T))^4 -  1641200890885920(\Gamma_1(T))^2 + 14770814323798008}{-5477180725633679(\Gamma_1(T))^4 + 197178506122812676(\Gamma_1(T))^2 - 1774606555105302716}. \nonumber
\end{align}
Note that $(0, \F_1(\Gamma_1(T)))$ satisfies $(C1)$ and $(C2)$. Hence using the same arguments as in the proof of Lemma \ref{infinitude-alpha-beta-gamma-lemma}, one knows that $\D_1(\Gamma_1(T)) \in \bZ_{29}$ for each $T \in \bQ$.

We see that
\begin{align*}
u := u_0 + A\D_1(\Gamma_1(T)) = u_0 = 166257 \ne 0
\end{align*}
for each $T \in \bQ$. Furthermore it follows from part $(2)$ of Lemma \ref{algebraic-family-hyperelliptic-curve-lemma} that
\begin{align*}
v = v_0 + B\D_1(\Gamma_1(T)) = \G_1(\F_1(\Gamma_1(T)))
\end{align*}
is defined (that is, not infinity) and non-zero for each $T \in \bQ$. Hence $(0, \F_1(\Gamma_1(T)))$ satisfies $(C3)$ in Lemma \ref{infinitude-alpha-beta-gamma-lemma} for each $T \in \bQ$.

Set
\begin{align*}
\alpha &:= \alpha_0 + 2p^2\D_1(\Gamma_1(T)) = 7 + 1682\D_1(\Gamma_1(T)) = \D_1^{\ast}(\Gamma_1(T)), \\
\beta  &:= \beta_0 = 261, \\
\gamma &:= \gamma_0 = 15,
\end{align*}
where $\D_1^{\ast}(\Gamma_1(T))$ is defined by $(\ref{Definition-The-rational-function-D-1-ast-of-Gamma-1})$ in Example \ref{Example-The-first-rational-function-D1-related-to-the-construction-of-algebraic-families}. Recall from $(\ref{Definition-The-rational-function-D-1-ast-of-Gamma-1})$ in Example \ref{Example-The-first-rational-function-D1-related-to-the-construction-of-algebraic-families} that
\begin{align*}
\D_1^{\ast}(\Gamma_1(T)) = \frac{\Sigma_{1, 1}(T)}{\Sigma_{1, 2}(T)},
\end{align*}
where $\Sigma_{1, 1}(T), \Sigma_{1, 2}(T)$ are defined by $(\ref{Equation-Sigma-1-1})$, $(\ref{Equation-Sigma-1-2})$, respectively.

By Lemma \ref{infinitude-alpha-beta-gamma-lemma}, we know that $(\alpha, \beta, \gamma)$ satisfies $(A1)-(A5)$ and $(B1)$ for each $T \in \bQ$. By $(\ref{eqn-P})$ and $(\ref{eqn-Q})$, we know that
\begin{align*}
P &= 29(7 + 1682\D_1(\Gamma_1(T)))^2 + 123192 = 29(\D_1^{\ast}(\Gamma_1(T)))^2 + 123192, \\
Q &= Q_0 = -63945.
\end{align*}
It is not difficult to see that for each $T \in \bQ$, the curve $\cC_{n, T, (29, 1, 3)}^{(7, 261, 15)}$ defined in part $(i)$ of Theorem \ref{algebraic-family-hyperelliptic-curve-hp-theorem} is the smooth projective model of the affine curve defined by $(\ref{eqn-hyperelliptic-curve-C})$.

By Corollary \ref{Corollary-The-separability-of-the-first-algebraic-family-of-hyperelliptic-curves}, we know that $\cP_{1, T}(x)$ is separable for each $T \in \bQ$, and hence we deduce that condition $(S)$ in Theorem \ref{hyperelliptic-curves-theorem} is true. Since $\D_1(\Gamma_1(T)) \in \bZ_{29}$ for each $T \in \bQ$, we see that
\begin{align*}
-2\left(\dfrac{\gamma}{\alpha}\right)^2 \equiv 21 \pmod{29}.
\end{align*}
Since $n \in \A_1$, we deduce that
\begin{align*}
n \not\equiv -2\left(\dfrac{\gamma}{\alpha}\right)^2 \pmod{29},
\end{align*}
and thus $(A6)$ holds if $n$ is odd and $(B2)$ holds if $n \equiv 2 \pmod{4}$. By Corollary \ref{hyperelliptic-curves-hp-corollary} and Corollary \ref{n-even-hyperelliptic-curve-hp-corollary}, we deduce that for each $n \in \A_1$ and each $T \in \bQ$, $\cC_{n, T, (29, 1, 3)}^{(7, 261, 15)}$ satisfies HP1 and HP2.

We now prove that part $(ii)$ of Theorem \ref{algebraic-family-hyperelliptic-curve-hp-theorem} is true. We will use the same notation as in the proof of part $(i)$ as long as they do not cause any confusion. We will use the same arguments as in the proof of part $(i)$ to construct an algebraic family of hyperelliptic curves of genus $n$ satisfying HP1 and HP2 for each $n \in \A_2$.

Let $n$ be any integer such that $n \in \A_2$. Let $(p, b, d, \alpha_0, \beta_0, \gamma_0) = (29, 1, 3, 133, 29, 27)$. We see that $\bar{\beta_0} = 1$. One can check that the sextuple $(p, b, d, \alpha_0, \beta_0, \gamma_0)$ satisfies $(A1)-(A5)$ and $(B1)$. Indeed $(A1), (A2), (A4), (A5)$ and $(B1)$ are obvious. It remains to prove that the sextuple satisfies $(A3)$. By $(\ref{eqn-P})$ and $(\ref{eqn-Q})$, we know that
\begin{align*}
P_0 &= 472381, \\
Q_0 &=  1711.
\end{align*}
The conic $\cQ_1^{(133, 29, 27)}$ in $(A3)$ of Theorem \ref{hyperelliptic-curves-theorem} defined by
\begin{equation*}
\cQ_1^{(133, 29, 27)} : 29U^2 - V^2 - 23439072839T^2 = 0
\end{equation*}
has a point $(u_0, v_0, t_0) = (728799, 3613777, 10)$, and thus $(p, b, d, \alpha_0, \beta_0, \gamma_0)$ satisfies $(A3)$.

Let $\S := \{2, 29\}$, and let $\C_2(T) \in \bQ(T)$ be the rational function defined by the same equation $(\ref{eqn-C})$ of $C$ with $(0, T)$ in the role of $(A, B)$, that is,
\begin{align*}
\C_2(T) = \dfrac{- 2v_0T - 4p^3\alpha_0\beta_0t_0^2Q_0}{T^2 + 4p^5\beta_0t_0^2Q_0} = \dfrac{-7227554T - 64380394481200}{T^2 + 407097080892400}.
\end{align*}
Let $\G_2(T) \in \bQ(T)$ be the rational function defined by
\begin{equation}
\label{G2(T)-rational-function-eqn}
\G_2(T) = v_0 + T\C_2(T) = \frac{-3613777T^2 - 64380394481200T + 1471158067696094594800}{T^2 + 407097080892400}.
\end{equation}
Since the numerator and denominator of $\G_2(T)$ are irreducible polynomials over $\bQ$, the set $\Z_2$ of rational zeroes and poles of $\G_2(T)$ is empty. Hence applying Lemma \ref{algebraic-family-hyperelliptic-curve-lemma} for the triple $\left(\S, \G_2(T), \Z_2 \right)$, we know that $\F_2(T)$ satisfies $(1)$ and $(2)$ in Lemma \ref{algebraic-family-hyperelliptic-curve-lemma}, where $\F_2(T)$ is the rational function defined by $(\ref{rational-function-F(t)-empty-case})$ in \textit{Case 2} of Lemma \ref{algebraic-family-hyperelliptic-curve-lemma} with $(p_0, \epsilon) = (3, 2)$ and $\left(\S, \G_2(T), \Z_2 \right)$ in the role of $(\S, \G(T), \Z)$, that is,
\begin{align*}
\F_2(T) := 1 + \dfrac{4\left(\prod_{l \in \S, l \ne 2}l \right)}{T^2 - p_0^2\epsilon} = 1 + \dfrac{116}{T^2 - 18} = \dfrac{T^2 + 98}{T^2 - 18}.
\end{align*}

Let $\Gamma_2(T) \in \bQ(T)$ be the rational function defined by $(\ref{Definition-The-definition-of-the-rational-function-Gamma2})$ in Example \ref{Example-The-second-rational-function-D2-related-to-the-construction-of-algebraic-families}. Recall that
\begin{align*}
\Gamma_2(T) := \dfrac{121}{T^2 - 63}.
\end{align*}
It is known that $\Gamma_2(T)$ is well-defined for each rational number $T \in \bQ$.

Let $(A, B) = (0, \F_2(\Gamma_2(T)))$. By Lemma \ref{algebraic-family-hyperelliptic-curve-lemma}, we know that $(0, \F_2(T))$ satisfies $(C1)$ and $(C2)$ in Lemma \ref{infinitude-alpha-beta-gamma-lemma} for each $T \in \bQ$. Thus it follows that $(A, B) = (0, \F_2(\Gamma_2(T)))$ satisfies $(C1)$ and $(C2)$ in Lemma \ref{infinitude-alpha-beta-gamma-lemma} for each $T \in \bQ$.

Let $\D_2(\Gamma_2(T))$ be the rational function in $\bQ(T)$ defined by the same equation $(\ref{eqn-C})$ of $C$ with $(0, \F_2(\Gamma_2(T)))$ in the role of $(A, B)$, that is,
\begin{align}
\label{D2-eqn-algebraic-family}
&\D_2(\Gamma_2(T)) := \C_2(\F_2(\Gamma_2(T))) \\
&= \frac{-64380401708754 (\Gamma_2(T))^4 + 2317693623118880 (\Gamma_2(T))^2 - 20859235062503544}{407097080892401 (\Gamma_2(T))^4 - 14655494912126204 (\Gamma_2(T))^2 + 131899454209147204}. \nonumber
\end{align}
Note that $(0, \F_2(\Gamma_2(T)))$ satisfies $(C1)$ and $(C2)$. Hence using the same arguments as in the proof of Lemma \ref{infinitude-alpha-beta-gamma-lemma}, one knows that $\D_2(\Gamma_2(T)) \in \bZ_{29}$ for each $T \in \bQ$.

We see that
\begin{align*}
u := u_0 + A\D_1(\Gamma_1(T)) = u_0 = 728799 \ne 0
\end{align*}
for each $T \in \bQ$. Furthermore it follows from $(2)$ of Lemma \ref{algebraic-family-hyperelliptic-curve-lemma} that
\begin{align*}
v = v_0 + B\D_2(\Gamma_2(T)) = \G_2(\F_2(\Gamma_2(T)))
\end{align*}
is defined (that is, not infinity) and non-zero for each $T \in \bQ$. Hence $(0, \F_2(\Gamma_2(T)))$ satisfies $(C3)$ in Lemma \ref{infinitude-alpha-beta-gamma-lemma} for each $T \in \bQ$.

Set
\begin{align*}
\alpha &:= \alpha_0 + 2p^2\D_2(\Gamma_2(T)) = 133 + 1682\D_2(\Gamma_2(T)) = \D_2^{\ast}(\Gamma_2(T)), \\
\beta  &:= \beta_0 = 29, \\
\gamma &:= \gamma_0 = 27,
\end{align*}
where $\D_2^{\ast}(\Gamma_2(T))$ is defined by $(\ref{Definition-The-rational-function-D-2-ast-of-Gamma-2})$ in Example \ref{Example-The-second-rational-function-D2-related-to-the-construction-of-algebraic-families}. Recall from $(\ref{Definition-The-rational-function-D-2-ast-of-Gamma-2})$ in Example \ref{Example-The-second-rational-function-D2-related-to-the-construction-of-algebraic-families} that
\begin{align*}
\D_2^{\ast}(\Gamma_2(T)) = \frac{\Sigma_{2, 1}(T)}{\Sigma_{2, 2}(T)},
\end{align*}
where $\Sigma_{2, 1}(T), \Sigma_{2, 2}(T)$ are defined by $(\ref{Equation-Sigma-2-1})$, $(\ref{Equation-Sigma-2-2})$, respectively.

By Lemma \ref{infinitude-alpha-beta-gamma-lemma}, we know that $(\alpha, \beta, \gamma)$ satisfies $(A1)-(A5)$ and $(B1)$ for each $T \in \bQ$. By $(\ref{eqn-P})$ and $(\ref{eqn-Q})$, we know that
\begin{align*}
P &= 29(133 + 1682\D_2(T))^2 - 40600 = 29(\D_2^{\ast}(\Gamma_2(T)))^2 - 40600, \\
Q &= Q_0 = 1711.
\end{align*}
It is not difficult to see that for each $T \in \bQ$, the curve $\cC_{n, T, (29, 1, 3)}^{(133, 29, 27)}$ defined in part $(ii)$ of Theorem \ref{algebraic-family-hyperelliptic-curve-hp-theorem} is the smooth projective model of the affine curve defined by $(\ref{eqn-hyperelliptic-curve-C})$.

By Corollary \ref{Corollary-The-separability-of-the-second-algebraic-family-of-hyperelliptic-curves}, we know that $\cP_{2, T}(x)$ is separable for each $T \in \bQ$, and hence we deduce that condition $(S)$ in Theorem \ref{hyperelliptic-curves-theorem} is true. Since $\D_2(\Gamma_2(T)) \in \bZ_{29}$ for each $T \in \bQ$, we see that
\begin{align*}
-2\left(\dfrac{\gamma}{\alpha}\right)^2 \equiv 8 \pmod{29}.
\end{align*}
Since $n \in \A_2$, we deduce that
\begin{align*}
n \not\equiv -2\left(\dfrac{\gamma}{\alpha}\right)^2 \pmod{29},
\end{align*}
and thus $(A6)$ holds if $n$ is odd, and $(B2)$ holds if $n \equiv 2 \pmod{4}$. By Corollary \ref{hyperelliptic-curves-hp-corollary} and Corollary \ref{n-even-hyperelliptic-curve-hp-corollary}, we deduce that for each $n \in \A_2$ and each $T \in \bQ$, $\cC_{n, T, (29, 1, 3)}^{(133, 29, 27)}$ satisfies HP1 and HP2.

\end{proof}

\begin{remark}
\label{Remark-Existence-of-algebraic-families-of-curves-of-genus-n-with-n-divisible-by-4}

It seems that the approach used in this paper to construct algebraic families of hyperelliptic curves of genus $n$ violating the Hasse principle does not work if $n \equiv 0 \pmod{4}$. The reason is that the hyperelliptic curves defined in Theorem \ref{hyperelliptic-curves-theorem} might fail to be locally solvable at $2$ if the genus of these curves is divisible by four, which in turn comes from the fact that the hyperelliptic curves in Theorem \ref{hyperelliptic-curves-theorem} lie on the threefolds $\cY$ in Theorem \ref{threefold-hp-theorem} and the local Azumaya invariant at $2$ of the threefolds $\cY$ is $1/2$.

In \cite{dq-Algebraic-families-of-curves-with-n-divisible-by-4}, the author constructs threefolds parameterized by the primes $p \equiv 1 \pmod{8}$ that are counterexamples to the Hasse principle. Furthermore it is shown in \cite{dq-Algebraic-families-of-curves-with-n-divisible-by-4} that there exists an odd prime $q$ such that the local Azumaya algebra at $q$ of the threefolds is $1/2$ and the local Azumaya algebra invariant at $l$ of the threefolds is zero for any prime $l \ne q$. For any hyperelliptic curve $\cC$ of genus $n$ lying on the threefolds in \cite{dq-Algebraic-families-of-curves-with-n-divisible-by-4}, since the local Azumaya algebra invariant at $2$ of the threefolds is zero, the local solvability at $2$ of $\cC$ should not depend on whether $n$ is divisible by $4$. Hence using the approach in this paper, it seems promising that there should exist algebraic families of curves of arbitrary genus violating the Hasse principle. The author hopes to finish this work in the near future.

\end{remark}

\section{Appendix}

\subsection{Proof of Lemma \ref{bright-corn-lemma}}
\label{Subsection-Proof-of-the-Bright-Corn-Lemma}

In this subsection, we sketch a proof of Lemma \ref{bright-corn-lemma}. The proof presented here follows closely from \cite[Proposition 4.17]{bright-Thesis} and \cite[Proposition 2.2.3]{corn-PLMS}. We begin by recalling some facts about the Tate cohomology groups.

Let $G$ be a finite group, and let $A$ be a $G$-module. The \textit{Tate cohomology groups $\widehat{H}^n(G, A)$ with negative and positive exponents} \cite[Chapter VIII]{serre-Local-Fields} are defined by
\begin{align*}
\widehat{H}^n(G, A) :=
\begin{cases}
H^n(G, A) \; &\text{if $n \ge 1$,}\\
A^G/NA \; &\text{if $n = 0$,} \\
_{N}A/I_GA \; &\text{if $n = -1$,} \\
H_{n - 1}(G, A) \; &\text{if $n \le - 2$,}
\end{cases}
\end{align*}
where $A^G$ is the submodule of $A$ consisting of the elements fixed by $G$, $N : A \rightarrow A$ is the endomorphism given by $Na = \sum_{g \in G}ga$ for all $a \in A$, $_{N}A$ is the kernel of $N$, and $I_G$ is the \textit{augmentation ideal} of the group algebra $\bZ[G]$, i.e., it consists of all linear combinations of the elements $g - 1$ with $g \in G$.

We now proceed to prove Lemma \ref{bright-corn-lemma}. We have the exact sequence
\begin{align}
\label{Equation-The-exact-sequence-of-Div-and-Pic}
0 \rightarrow F(\cX_L)^{\times}/L^{\times} \rightarrow \text{Div}(\cX_L) \rightarrow \text{Pic}(\cX_L) \rightarrow 0,
\end{align}
where $\cX_L = \cX \times_{F} L$ and $F(\cX_L)$ denotes the function field of $\cX_L$. Let $\text{Gal}(L/F)$ denote the Galois group of the extension $L/F$. It is well-known \cite[Corollary to Proposition 6, page 133]{serre-Local-Fields} that the Tate cohomology groups with exponents $n$ depend only on the parity of $n$. Hence it follows from the exact sequence $(\ref{Equation-The-exact-sequence-of-Div-and-Pic})$ that we have a commutative diagram
\begin{align}
\label{Equation-The-commutative-diagram-of-Div-and-Pic}
\begin{CD}
\widehat{H}^{-1}(\text{Gal}(L/F), \text{Pic}(\cX_L)) @>>> \widehat{H}^{0}(\text{Gal}(L/F), F(\cX_L)^{\times}/L^{\times})    @>>>  \widehat{H}^{0}(\text{Gal}(L/F), \text{Div}(\cX_L))     \\
@VVV             @VVV       @VVV      \\
H^{1}(\text{Gal}(L/F), \text{Pic}(\cX_L))  @>>> H^{2}(\text{Gal}(L/F), F(\cX_L)^{\times}/L^{\times})     @>>> H^{2}(\text{Gal}(L/F), \text{Div}(\cX_L)),
\end{CD}
\end{align}
where the vertical arrows are isomorphisms. By \cite[Chapter V]{deuring-Algebra}, we know that the cyclic algebra $(L/F, f)$ corresponds to the element $f \in H^{2}(\text{Gal}(L/F), F(\cX_L)^{\times}/L^{\times})$, and it thus follows from the diagram $(\ref{Equation-The-commutative-diagram-of-Div-and-Pic})$ that $(L/F, f)$ corresponds to the element $f \in \widehat{H}^{0}(\text{Gal}(L/F), F(\cX_L)^{\times}/L^{\times})$.

By the definition of the Tate cohomology groups, we see that
\begin{align*}
\widehat{H}^{0}(\text{Gal}(L/F), \text{Div}(\cX_L)) = \text{Div}(\cX)/N \text{Div}(\cX_L).
\end{align*}
Thus our contentions follow immediately from the commutative diagram $(\ref{Equation-The-commutative-diagram-of-Div-and-Pic})$.

\section*{Acknowledgements}

I was supported by a postdoctoral fellowship in the Department of Mathematics at University of British Columbia during the time when the paper was prepared.

\end{document}